\documentclass[12pt]{amsart}\usepackage{mathtools,amssymb,latexsym,graphics,enumerate}
\usepackage[mathscr]{eucal}
\usepackage{amsmath,amsfonts,amsthm,amssymb,bbm,mathtools}
\usepackage{color}
\usepackage{cite}
\usepackage{hyperref}
\usepackage{comment}
\usepackage[left=1in,right=1in,top=1in,bottom=1in]{geometry}

\usepackage{graphicx}
\usepackage{subcaption}
\usepackage{float}
\usepackage[dvipsnames]{xcolor}

\numberwithin{equation}{section}

\newcommand{\rr}{\mathbb{R}}

\newcommand{\lan}{\langle}
\newcommand{\ran}{\rangle}
\newcommand{\be}{\begin{eqnarray*}}
\newcommand{\bel}{\begin{eqnarray}}
\newcommand{\ee}{\end{eqnarray*}}
\newcommand{\eel}{\end{eqnarray}}
\newcommand{\ba}{\begin{aligned}}
\newcommand{\ea}{\end{aligned}}
\newcommand{\de}{\Delta}
\newcommand{\al}{\alpha}
\newcommand{\na}{\nabla}
\newcommand{\ep}{\epsilon}
\newcommand{\nq}{{\neq}}

\newcommand{\myr}[1]{{#1}}

\newcommand{\bp}{{\mathbf{p}}}
\newcommand{\bx}{{\mathbf{x}}}
\newcommand{\by}{{\mathbf{y}}}
\newcommand{\bk}{{\mathbf{k}}}
\newcommand{\bv}{{\mathbf{v}}}

\newcommand{\fz}{\langle f\rangle}
\newcommand{\mf}{\mathfrak}

\newcommand{\myb}[1]{}
\newcommand{\myh}[1]{}

\newcommand{\pa}{\partial}
\newcommand{\wh}{\widehat}
\newcommand{\wt}{\widetilde}
\newcommand{\lf}{\left}
\newcommand{\rg}{\right}
\newcommand{\te}{\theta}

\newtheorem{theorem}{Theorem}
\newtheorem{cor}{Corollary}

\newtheorem{lem}{Lemma}

\newtheorem{remark}{Remark}

\numberwithin{remark}{section}
\numberwithin{cor}{section}
\numberwithin{lem}{section}
\numberwithin{theorem}{section}

\newcommand{\Torus}{{\mathbb T}}

\newcommand{\dss}{\displaystyle}

\newcommand{\n}{\ensuremath{\nonumber}}
\newcommand{\sima}[1]{{#1}}

\mathtoolsset{showonlyrefs}
\allowdisplaybreaks

\title{Mixing, Enhanced Dissipation and Phase Transition in the Kinetic Vicsek Model} 

\author[Mengyang Gu]{Mengyang Gu} 
\address{Department of Statistics and Applied Probability, University of California,  Santa Barbara,  CA 93106, USA}
\email{mengyang@pstat.ucsb.edu}
\author[Siming He]{Siming He}
\address{Department of Mathematics, University of South Carolina, Columbia, SC 29208, USA}
\email{siming@mailbox.sc.edu} 
\thanks{{\bf Acknowledgment:} M.G. acknowledges partial support from the National Science Foundation (NSF) under award number OAC-2411043. S.H. is partly supported by NSF grants DMS 2006660, 2304392, 2406293, and 2006372. 
We thank the anonymous referees for their comments and valuable suggestions.}
\begin{document}
\begin{abstract}
In this paper, we study the kinetic Vicsek model, which serves as a starting point to describe the polarization phenomena observed in the 
experiments of fibroblasts moving on liquid crystalline substrates, 
detailed in \cite{luo2022cell}. The long-time behavior of the kinetic equation is analyzed, revealing that, within specific parameter regimes, the mixing and enhanced dissipation phenomena stabilize the dynamics and ensure effective information communication among agents. Consequently, the solution exhibits features similar to those of a spatially-homogeneous system studied in \cite{FrouvelleLiu12}. As a result, we confirm the phase transition observed in Vicsek et al. \cite{Vicsek95} on the kinetic level.
\end{abstract} 
\maketitle
\section{Introduction}
Experiments suggest that 
common cells in muscle and connecting tissues, such as myoblasts and fibroblasts, exhibit an orientation alignment phenomenon induced by the weak influence of a molecularly aligned substrate (see, e.g., \cite{martella2019liquid,turiv2020topology,luo2022cell,GuFangLuo23,chen2024disodium}). As a result of this alignment process, the muscle fibers developed later on are ordered. 
However, a mathematically rigorous justification for this emergence of order remains open. Understanding the underlying mechanism of cell alignment is crucial for designing biomaterials that present functional properties of human organs in tissue engineering. The goal of the paper is to provide a theoretical framework to analyze this family of pattern formation processes in human tissue development.

The starting point of our discussion is the agent-based stochastic differential equation (SDE)-dynamics
\begin{align}\label{micro_eq}\begin{cases}
d{\bx}^i =v\bp(\theta^i)d t:=v(\cos(\theta^i),\sin(\theta^i))d t,   \\
\displaystyle d\theta^i=\frac{\kappa}{N}\sum_{j=1}^{N} \Phi(\bx^j-\bx^i)\Psi(\theta^j-\theta^i)dt+{\sqrt{2\nu}dB^i_t},\\
(\bx^i ,\theta^i)\big|_{t=0}=(\bx^i_0,\theta_0^i),\quad \bx^i\in {{\mathbb T}}^2,\, \,\theta^i\in {\mathbb T},\quad i\in \{1,2,..., N\}. 
\end{cases}
\end{align} 
Here, $\bx^i(t)$ and $\theta^i(t)$ denote the position and velocity direction of the $i$th agent (e.g. a fibroblast), respectively, while $N$ represents the total number of agents. To simplify the model, we assume that $N$ stays constant over time. We assume that all the agents move with the \sima{given} speed $v=v(t)>0$. In the experiment, due to the proliferation effect (which has not yet been incorporated into our model), the moving agents gradually slow down as time progresses, resulting in $v(t)$ being a decreasing function over time. Agents consistently adjust their velocity to match the dynamic average velocity of their neighbors, following specific `communication protocols'. The influence functions $(\Phi, \Psi)$ encode the spatial and angular dependencies of this averaging process. The parameter $\kappa$ signifies the strength of the alignment forces. Additionally, to model the randomness inherent in the experiment, we introduce the independent and identically distributed (i.i.d.) Brownian noise $\sqrt{2\nu}dB_t^i$, where $\nu$ represents the noise strength. Based on the experimental observation \cite{luo2022cell}, proliferation effect plays a crucial role in promoting alignment among the agents. Hence, it is reasonable to consider SDE models with an increasing population size $N$. This will be discussed in future work.

On a mesoscopic level, the following model characterizes the large population limit of the Vicsek dynamics \eqref{micro_eq}
\begin{subequations}\label{eq:bsc}
\begin{align}\label{eq:bsc_1}\begin{cases}
\pa_t f+v \bp\cdot \na_\bx f+\kappa \pa_\theta(fL[f])=\nu \pa_\theta^2 f,\quad \bp(\theta):=(\cos(\te),\sin(\te)),\\
f(t=0,\bx ,\theta)=f_{0}(\bx,\theta),\quad \bx\in {\mathbb T}^2,\quad \theta\in{\mathbb T}.\end{cases} 
\end{align}
The density function $f(t,\bx,\te)\geq 0$ captures the population distribution of the cells moving with  velocity $v\bp$ at position $\bx.$ 
Here $\bp(\theta)=(\cos(\theta), \sin(\theta))$  represents the orientation of the particle, and $v=v(t)$ indicates the speed. 
The alignment effect is encoded in the operator $L$, defined as
\begin{align}\label{Lf}
L[f](t,\bx,\te)=\int_{{\mathbb T}}\int_{{\mathbb T}^2} \Phi(\by-\bx)\Psi(\eta-\te) f(t,\by,\eta)d\by d\eta.
\end{align}
\end{subequations}
Here, the integral kernel is a product of the \emph{spatial influence function} $\Phi$ and the \emph{angular influence function} $\Psi$. We sketch the derivation of the equation \eqref{eq:bsc} from the agent-based dynamics in Appendix \ref{App:Derivation}. The argument is from the classical literature \cite{BolleyCanizoCarrillo12}. A few features during the alignment process of fibroblasts on liquid crystalline substrates were found in \cite{GuFangLuo23}. First, the agents only interact with their direct neighbors in a small spatial region. Second, fibroblasts traveling in opposite directions can glide past each other without influencing each other’s velocity, and similar behaviors are also found in behaviors of epithelial cells during morphogenesis \cite{guillot2013mechanics,bi2014energy}. 
Consequently, the spatial influence function $\Phi$ is local, and the angular influence function $\Psi$ is heterogeneous. Finally, we observe that the divergence structure of the equation \eqref{eq:bsc} guarantees conservation of the total mass of $f$ (as long as the solution is regular enough),
\begin{align}\label{Mass}
\|f(t)\|_{L_{\bx,\theta}^1}\equiv\|f_0\|_{L_{\bx,\theta}^1}.
\end{align}
\sima{By normalizing the density $\wt {f}:=f/\|f_0\|_{L_{\bx,\te}^1}$ and redefining the alignment parameter $\wt\kappa:=\kappa\|f_0\|_{L_{\bx,\te}^1}$, one can reduce the system \eqref{eq:bsc_1} to the special case where $\|\wt f\|_{L_{\bx,\theta}^1}\equiv 1$. Hence, without loss of generality, we assume that  $\|f_0\|_{L_{\bx,\theta}^1}=1$ throughout the paper.  } 

The agent-based system \eqref{micro_eq} has been extensively studied since Vicsek's pioneering work \cite{Vicsek95}. Numerically, it has been observed that the model exhibits a phase transition phenomenon. When the strength of the noise is significant, 
randomness dominates the agents' collective behavior. However, the alignment phenomenon emerges at small noise levels. To understand this critical threshold for phase transition, physicists \cite{TonerTu95, TonerTu98, Ihle11} and mathematicians \cite{DegondMotsch07,BolleyCanizoCarrillo12} have derived various kinetic and hydrodynamic limits for the Vicsek model. \sima{In the \emph{spatially homogeneous} case where the solution $f$ to \eqref{eq:bsc_1} is independent of the spatial $\bx$-variable, i.e., $f(t,\bx,\theta)=f(t,\te)$,  the phase transition phenomenon is justified in the works \cite{FrouvelleLiu12,DegondFrouvelleLiu15}. However, in the \emph{spatially inhomogeneous case} where the $\bx$-variable is involved, much less is known. To the best of our knowledge, there is no rigorous mathematical derivation for the phase transition of the full kinetic Vicsek model \eqref{eq:bsc_1}. Our main goal in this paper is to develop a mathematical framework to derive the phase separation in the \emph{spatially inhomogeneous scenario}.}

To gain further insight into the kinetic Vicsek model \eqref{eq:bsc}, \sima{we introduce a related mathematical model, i.e., the Cucker-Smale (CS) flocking model \cite{CS07}}. The CS model has been extensively studied, and the literature centered around it is vast. We refer the readers to the work \cite{MotschTadmor11, CCTT16, TanTadmor, HaLiu09, HaTadmor08, HeTadmor17, ShvydkoyTadmorI, ShvydkoyTadmorII, ShvydkoyTadmorIII, Shvydkoy18, Shvydkoy21, DietertShvydkoy19, DietertShvydkoy21, DKRT17} for further discussions. In this model, the speed of the agents is no longer restricted, and no noise is present. The primary challenges in analyzing the CS model involve understanding the long-time dynamics of the model under a \emph{local communication protocol}, where agents only interact with close neighbors. Even within a bounded domain, the agents might fail to align due to the emergence of isolated ``communities'' that do not have efficient information interchange. This isolation typically occurs when vacuum regions form among the agents, and hence, the minimum of the density function approaches zero. It is observed in \cite{ShvydkoyTadmor20Top} that the critical ingredient guaranteeing unconditional flocking for systems with a local communication protocol is the slow decay of the density minimum.  However, this slow decay is challenging to justify mathematically in general. With an appropriately designed \emph{topological communication protocol}, the authors are able to justify the slow decay of density minimum and, hence, the unconditional flocking phenomenon. Incorporating noise is another way to regularize the long-time dynamics of the CS dynamics on the kinetic level, \cite{Shvydkoy22hypo}. 

Based on the discussion above, we specify the basic assumptions on the spatial/angular influence function pair $(\Phi,\Psi)$. The spatial influence function $\Phi\in C^\infty({\mathbb T}^2)$ is even with respect to the argument:
\begin{align}
\Phi(\bx-\by)=\Phi(\by-\bx),\quad \forall\bx,\by\in {\mathbb T}^2. \label{phi_sym}
\end{align} 
We emphasize that no additional structural assumptions are imposed on $\Phi$. Hence, the function $\Phi$ can have a small support, corresponding to the phenomenon that the agents only interact with their close neighbors. 
On the other hand, we consider angular influence functions $\Psi\in C^\infty({\mathbb T})$ satisfying the following structural assumptions:
\begin{align} \Psi(\theta )=\sin(\theta )\psi(\theta ),\quad0\leq \psi  \in C^\infty ({\mathbb T}),\quad \int_{-\pi}^\pi\Psi(\theta)d\theta=0.\label{Psi_strc}
\end{align}
Here, $\psi$ is a positive, smooth, even function with respect to the argument $\theta$. We note that if $\theta=\pi$ (i.e., when two agents are moving in opposite directions), the interaction between them is zero. 

If the solution $f$ of \eqref{eq:bsc} does not depend on the spatial variable $\bx$, then it is a solution to the following spatially homogeneous equation \sima{(see, e.g., \cite{FrouvelleLiu12})}:
\begin{align}\label{effctv_dym}
\pa_t  g-\kappa\lf(\int_{{\mathbb T}^2}\Phi d\bx\rg)\pa_\theta\big( g(\Psi * g)\big)=\nu\pa_\theta^2  g,\quad g(t=0,\theta)=g_0(\theta).
\end{align}
Here, the notation $*$ represents the angular convolution. If the integral $\int_{{\mathbb T}^2}\Phi d\bx$ is positive but different from $1$, one can introduce the effective alignment parameter $\wt\kappa:=\kappa \displaystyle\int \Phi d\bx$ to simplify the equation. For simplicity, we assume that \begin{align}\label{normalize_Phi}
\int_{\Torus^2}\Phi d\bx= 1
\end{align}throughout the paper. We also highlight that our result in the paper will not conflict with experimental conditions in the sense that if the support of smooth
 influence function $\Phi$ is too small, then the integral $\displaystyle \int\Phi d\bx$ is tiny, which yields a small effective $\wt\kappa$.

To begin the analysis, we introduce a new quantity that guarantees effective information exchange among agents. Let us consider the $\bx$-average and remainder of the agent density $f$,
\begin{align}\label{x-avrg_rem}
\lan f\ran(t,\theta) =\frac{1}{|\Torus|^2}\sima{\int_{{\mathbb T}^2}}f(t,\bx,\theta)d\bx,\quad f_\nq(t,\bx,\theta)=f(t,\bx,\theta)-\lan f\ran(t,\theta).
\end{align} 
If the remainder is zero ($f_\nq\equiv 0$), the density function is homogeneous in the spatial variable $\bx$. When combined with the conservation of mass \eqref{Mass}, this constraint $f_\nq\equiv 0$ implies that the marginal density $\displaystyle\rho(t,\bx)=\int f(t,\bx,\te)d\te=\int \lan f\ran(t,\te) d\te=|\Torus|^2\|f\|_{L^1}=|\Torus|^2\|f_0\|_{L^1}$ is constant. Hence, one can find moving agents with equal probability across the spatial domain $\Torus^2.$ Hence, we expect that if the remainder $f_\nq$ decays quickly to zero in suitable norms, the system will rapidly converge to the spatially homogeneous state ($f_\nq\equiv 0$), making information exchange efficient within the system. From a mathematical analysis perspective, it is sufficient to consider the $L^2/H^{-1}$-norms of the remainder $f_\nq$, i.e., $\|f_\nq\|_{L^2}^2$ or $\|f_\nq\|_{H^{-1}}^2$.  

Now we identify the stabilization mechanisms in the system \eqref{eq:bsc} that guarantee the fast decay of the remainder $f_\nq$. To this end, we consider the simplified system $\eqref{eq:bsc_1}_{\kappa=0}$, referred to as the passive scalar equation:
\begin{align} \label{PS_key}
\pa_t\eta+v\bp\cdot\na_\bx\eta=\nu\pa_{\te}^2\eta.
\end{align}
The main idea for analyzing the nonlinear dynamics \eqref{eq:bsc} is to leverage two key stabilization effects in equation \eqref{PS_key}: the \emph{enhanced dissipation phenomenon} and the \emph{mixing phenomenon}. To illustrate these concepts, we introduce a further simplified model with small viscosity $0<\nu\ll 1$:
\begin{align}\label{ps}
\pa_t h+\sin(\theta)\pa_x h=\nu\pa_\theta^2 h,\quad h(t=0)=h_0,\quad \overline{h_0}:=\int_{{\mathbb T}}h_0 dx=0,\quad (x,\theta)\in {\mathbb T}^2.
\end{align}
A classical energy estimate shows that the $L^2$ norm of the solution, $\|h-\overline{h}\|_{L^2}$, decays on the heat dissipation time scale of $\mathcal{O}(\nu^{-1})$. However, it turns out that the remainder $h_\nq$ (see \eqref{x-avrg_rem})  decays on a much faster time scale. The following estimate is derived in various works \cite{BCZ15,Wei18,AlbrittonBeekieNovack21} and proven to be sharp in \cite{CotiZelatiDrivas19},
\begin{align}\label{ED_shear}
\|h_\nq(t)\|_{L^2}\leq C\|h_{\nq}(0)\|_{L^2}\exp\lf\{-\delta\nu^{1/2} t\rg\},\quad\forall t\in[0,\infty).
\end{align}
We observe that the remainder decays on a time scale of $\mathcal{O}(\nu^{-1/2})$, which is significantly shorter than the heat dissipation time scale $\mathcal{O}(\nu^{-1})$ in the parameter regime $0<\nu\ll 1$. This is known as the \emph{enhanced dissipation phenomenon}. The study of the enhanced dissipation phenomenon dates back to Lord Kelvin \cite{Kelvin87} and has attracted much attention in recent years, see, e.g., \cite{BCZ15,Wei18,AlbrittonBeekieNovack21,CobleHe23,CotiZelatiDrivas19,ElgindiCotiZelatiDelgadino18,
FengIyer19,CKRZ08}. \sima{The key enhanced dissipation estimate for $\eqref{PS_key}_{v\equiv1}$ was derived in the papers \cite{FengShiWang22,AlbrittonOhm22, CotiZelatiDietertGerardVaret22,CotiZelatiDietertVaret24}, where the authors study the Patlak-Keller-Segel model for chemotaxis (see, e.g., \cite{Patlak,KS}) and the  Saintillan-Shelley model for active swimmers (see, e.g.,  \cite{SaintillanShelley08, SaintillanShelley12}). The main conclusion from these works is that a suitable modification of the estimate \eqref{ED_shear} still persists for the linear dynamics \eqref{PS_key}, and the estimate yields deep insights into the long-time behavior of these nonlinear models. The hypocoercivity method \cite{BCZ15,villani2009} and the resolvent method \cite{Wei18} were applied to develop the estimate for \eqref{PS_key}. } 

Another important stabilization mechanism associated with the passive scalar equation \eqref{ps} is the \emph{mixing phenomenon}, which captures the fast, viscosity-independent decay of the negative Sobolev norms of the solutions. \sima{As illustrated in \cite{ElgindiCotiZelatiDelgadino18,FengIyer19}, enhanced dissipation and mixing phenomena are closely connected. They have found applications in fluid mechanics  \cite{BM13,Jia22,MasmoudiZhao22,IJ23,Zillinger2014,BGM15I,BGM15II,ChenLiWeiZhang18,WeiZhangZhao20,LiZhao21}, plasma physics \cite{MouhotVillani11,IonescuPausaderWangWidmayer24,Bedrossian21,BedrossianMasmoudiMouhot16,GrenierNguyenRodnianski21,ChaturvediLukNguyen23}, mathematical biology \cite{KiselevXu15,BedrossianHe16,He22,HuKiselevYao23,HuKiselev23}, and various other areas \cite{FengFengIyerThiffeault20, CotiZelatiDolceFengMazzucato,GongHeKiselev21,HeKiselev22}.} 

In Appendix \ref{App_A}, we will summarize the analysis done in \cite{AlbrittonOhm22,CotiZelatiDietertGerardVaret22,CotiZelatiDietertVaret24} and show that enhanced dissipation and mixing persist in the \emph{linear} equation \eqref{PS_key}. \sima{With all these concepts introduced, we are ready to present our first theorem, which captures the \emph{nonlinear} enhanced dissipation and mixing phenomena in the dynamics  \eqref{eq:bsc}. }
\begin{theorem}[Spatial Homogenization]\label{thm_1}
Consider solutions to equation \eqref{eq:bsc} subject to initial condition $0< f_0\in C^\infty({\mathbb T}^3)$. Assume that the speed profile $v(\cdot)\in C^\infty(\rr_+)$ takes values in $(1/2,1]$ and that the $C^\infty$ smooth influence functions $\Phi$ and $\Psi$ satisfy the structural conditions \eqref{phi_sym}, \eqref{Psi_strc}.  \sima{Further assume that the parameters $\kappa,\nu$ take values in $(0,1]$.} Then the following two claims hold.

\noindent
{\bf a) Enhanced Dissipation:} There exists a threshold $\mathfrak a=\mathfrak a(\Phi,\Psi, {\|f_0\|_{L^2}})>0$ 
 such that if  $0<\kappa\leq  \nu^{5/6+\wt{\gamma}}\leq \mathfrak{a}$ ($\wt{\gamma}>0$), then the following  estimate holds\begin{subequations}
\begin{align}\label{nl_ED}
\|f_\nq(t)\|_{L^2}\leq C_1\|f_{0;\nq}\|_{L^2}\exp\lf\{-\delta \nu^{1/2}t\rg\},\quad \forall t\in[0,\infty).
\end{align}
Here, $C_1\geq 1,\, \delta\in(0,1)$ are universal constants. Moreover, the $\bx$-average $\lan f\ran$ is bounded as follows
\begin{align}
 \|\lan f\ran(t)\|_{L^2}\leq C_2 (1+\|\lan f_0\ran\|_{L^2})\lf(1+\frac{\kappa^{1/2}}{\nu^{1/2}}\rg),\quad \forall t\in[0,\infty).   \label{nl_z_mod}
\end{align}\end{subequations}
Here, $C_2$ is a universal constant.

\noindent
{\bf b) Mixing:} If one assumes that the agent speed $v(t)\equiv 1$ and there exists a universal constant $C_\dagger\geq 1$ such that $0<\kappa\leq C_\dagger\nu$, the following estimate holds
\begin{align}
\|f_\nq(t)\|_{\dot H^{-1}}\leq C_3 \frac{  \|f_{ 0;\nq} \|_{H^1}}{ t^{1/2} },\quad \forall t\leq \delta^{-1}\nu^{-1/2}.    \label{nonlin_ID}
\end{align}
Here, the constant $C_3=C_3(\Psi, \Phi,\delta^{-1}, C_\dagger,\|\lan f_0\ran\|_{L^2})$.
\end{theorem}
\begin{remark}
Here, we do not impose any constraint on the support of the influence functions $(\Phi,\Psi)$. Hence, the agents might only interact with close neighbors. However, the migrating agents will rapidly spread out in space thanks to the transport-induced enhanced dissipation effect. This is the biological interpretation of \eqref{nl_ED}.  
\end{remark}
\begin{remark}
Our proof also works for the case where there exists `anti-alignment' such that $\int \Phi d\bx =0$. In this case, the solution $f$ will relax to a constant state as time approaches infinity.
\end{remark}
\begin{remark}
   The estimate \eqref{nonlin_ID} suggests that ``hydrodynamic quantities" of the form $\displaystyle{\int f(t,\bx,\te) g(\te)d\te}$, where $g(\cdot) \in C^2$, converge to spatially homogeneous states $\displaystyle{\int \lan f\ran(t,\te) g(\te)d\te}$ (in some weak spaces) at a rate that is independent of $\nu$ until the time scale $\mathcal{O}(\nu^{-1/2})$.
\end{remark}
As a result of Theorem \ref{thm_1}, we observe that the density $f$ quickly homogenizes in the $\bx$ direction, and the dynamics simplifies to that of \eqref{effctv_dym}. Hence, we examine the energy structure and stationary states of \eqref{effctv_dym}. To present the main results, we define the primitive function:
\begin{align}\label{U}
\mathbb{U}(\theta)=\int_{-\pi}^\theta\Psi(\eta)d\eta=\int_{-\pi}^\theta \sin(\eta)\psi(\eta)d\eta\in C^\infty({\mathbb T}).
\end{align}
We also note that if $\psi\geq 0$,  $\mathbb{U}(\theta)\leq 0$ for all $\theta\in {\mathbb T}$. For the sake of simplicity, we also consider the following concrete example $
\Psi_0,\, \mathbb{U}_0$:
\begin{align}\label{psiU0}
\Psi_0(\theta):=\sin(\theta) , \quad \mathbb{U}_0(\theta):=\int_{-\pi}^\theta\sin(\eta)d\eta=-1-\cos(\theta)\leq 0. 
\end{align}

Our second main theorem concerning the equation \eqref{effctv_dym} reads as follows. 
\begin{theorem}\label{thm_2}Consider the spatial homogeneous equation \eqref{effctv_dym} subject to conditions \eqref{phi_sym} and \eqref{Psi_strc}.

\noindent
a) Consider regular solutions $g\in C^3_{t,\theta}$ to \eqref{effctv_dym}. The free energy
\begin{align}\label{Free_energy}
F[g]:=\nu\int_{\mathbb T} g\log g d\theta +\frac{\kappa}{2}\iint_{{\mathbb T}\times {\mathbb T}} \mathbb{U}(\theta-w)g(\theta)g(w)dwd\theta.
\end{align} is decaying in time, i.e., 
\begin{align}\label{Fisher}
\frac{d}{dt}F[g]=&-\int g|\nu\pa_\theta\log  g +\kappa \Psi*g |^2=:-\mathcal{D}[g]\leq 0.
\end{align}

\noindent b) If
\begin{align}
\sup_{\ell\neq0}\frac{\kappa}{2\pi}|\wh {\mathbb{U}}(\ell)|<\nu,
\end{align}
the constant state $g\equiv \frac{1}{2\pi}$ is linearly stable. If there exists $\ell\in \mathbb{Z}\backslash\{0\}$ such that
\begin{align}
-\frac{\kappa}{2\pi}\Re\wh{\mathbb{U}}(\ell)>\nu,
\end{align}
the constant state $g\equiv \frac{1}{2\pi}$ is linearly unstable. 
\end{theorem}
As a simple corollary of this theorem, we partially recover the observation of \cite{FrouvelleLiu12}.
\begin{cor}
Assume $\Psi(\cdot)=\sin(\cdot)$. Then the following two statements hold
\begin{itemize}
\item If $\frac{\kappa}{\nu}<2$, the constant state is linearly stable;

\item If $\frac{\kappa}{\nu}>2$, the constant state is linearly unstable.
\end{itemize}
\end{cor}
If the angular influence function is sine, then much more is known for the spatial homogeneous problem \eqref{effctv_dym}. In the classical work \cite{FrouvelleLiu12}, the authors are able to show that the solution $g$ converges to an element in the family of Fisher-von Mises distributions which contains the energy minimizers of the free energy \eqref{Free_energy}. A key step in their proof is to derive a LaSalle principle.


For the full spatially inhomogeneous model \eqref{eq:bsc}, we are able to exploit the enhanced dissipation estimate \eqref{nl_ED} and partially recover the result of \cite{FrouvelleLiu12}. 
\begin{theorem}
\label{thm_3} Under the conditions of Theorem \ref{thm_1} {\bf a)}, the solutions $f$ to \eqref{eq:bsc_1}, \eqref{Lf} have the following asymptotic behaviors.

\noindent
{\bf a)} The solutions converge to a family of limiting configurations as $t$ approaches $\infty$ in the sense that
\begin{align}\label{convergence}
\lim_{t\rightarrow\infty} \inf_{G\in \mathcal{S}_\infty}\|f(t,\cdot)-G(\cdot)\|_{H_{\bx,\theta}^M}=0,\quad \forall M\in \mathbb{N}.
\end{align}
Here, the set $\mathcal{S}_\infty:=\{f\in C^\infty({\mathbb T}) |\mathcal{D}[f]=0\}$ is the set on which the Fisher information $\mathcal{D}$ \eqref{Fisher} vanishes.

\noindent
{\bf b)} There exists a constant $\mathfrak{b}=\mathfrak{b}(\Phi,\Psi)\in(0,1)$ such that, for the parameter regime $\kappa\leq \mathfrak b \nu\leq \mathfrak b^2$, the solutions $f$ converge to the constant states $f_\infty = \text{const.}$ 
as time approaches $+\infty$. 

\noindent
{\bf c)} Furthermore, if $\Psi(\cdot)=\sin(\cdot)$, the limiting state space $\mathcal{S}_\infty$ can be characterized by the ratio $\kappa/\nu$.
\begin{itemize}\item If the diffusion is stronger, i.e., $\kappa/\nu\leq 2$, the limiting state $\mathcal S_\infty:=\{G(\te)=\text{const.}\}$. 
\item If the alignment is stronger, i.e., $\kappa/\nu>2$, the limiting state $\mathcal S_\infty$ is characterized by a complex number  $r=|r|\exp\{i\arg r\}\in \mathbb{C}$, i.e., $\mathcal S_\infty:=\{G\in C^\infty|G(\te)=g_s^{(r)}(\te),\ \forall \te\in\Torus,\ g_s^{(r)}\text{ are defined in } \eqref{g_s} \}$. Here, $\arg r\in [-\pi,\pi]$, and $|r|$ satisfies  a compatibility condition \eqref{comp_cond} and takes two distinct values, i.e., $|r|\in\{r_1:=0,r_2:=r_2(\kappa/\nu)\}$.\end{itemize} 
\end{theorem}
\begin{remark}
Our theorem generalizes the result derived in \cite{FrouvelleLiu12} in the sense that the equation that we consider is spatially inhomogeneous, and the transport effect is dominant. Moreover, the spatial influence functions that we consider are allowed to be compactly supported in space. 
\end{remark}
The paper is organized as follows: in Section \ref{sec:ED_ID}, we derive Theorem \ref{thm_1}; in Section \ref{sec:eff_sym}, we derive Theorem \ref{thm_2} and Theorem \ref{thm_3}. 

\noindent
{\bf Notation: }Throughout the paper, the constants  $C$ depend on the norm $\|\Phi\|_{W_\bx^{M,\infty}},\ \|\Psi\|_{W_\te^{M,\infty}}$, $M\in \mathbb{N}$ and change from line to line. 

\section{Nonlinear \sima{Enhanced} Dissipation and Mixing}\label{sec:ED_ID}
In this section, we prove the main Theorem \ref{thm_1}. 
\begin{proof}[Proof of Theorem \ref{thm_1} {\bf a)}]
We divide the proof into three steps. The general plan is to apply the bootstrap argument. In {\bf Step \# 1}, we explicitly spell out the bootstrap assumptions and conclusions. In the latter steps, we prove each conclusion. 

\noindent{\bf Step \# 1: General setup.} First of all, we recall the definitions \eqref{x-avrg_rem} and decompose the solution $f$ to  \eqref{eq:bsc} as follows:\begin{subequations}
\begin{align}
\pa_t \lan f\ran +&\kappa\pa_\theta\lan L[f]f\ran=\nu\pa_{\theta}^2\lan f\ran,\quad \lan f\ran_0=\lan f_0\ran;\label{eq_average}
\\
\pa_t f_\nq+&v\bp\cdot \na_\bx f_\nq+\kappa\pa_\theta(L[f]f)_\nq=\nu\pa_{\theta}^2f_\nq,\quad f_{0;\nq}=(f_0)_\nq.\label{eq_nq}
\end{align}
\end{subequations}We first observe that \begin{align}\label{cons_mass}
\|\lan f\ran\|_{L_{\te}^1}=\frac{1}{2\pi}\|f\|_{L_{\bx,\theta}^1}=\frac{1}{2\pi}\|f_{0}\|_{L_{\bx,\theta}^1}.
\end{align}
 Next, we lay out the bootstrap assumptions. Assume that $[0,T_\star)$ is the maximal time interval such that the following hypotheses hold\begin{subequations}
\begin{align}
\|f_\nq(t)\|_{L^2}\leq 4e\|f_{\nq}(0)\|_{L^2}\exp\lf\{-\delta\nu^{1/2}t\rg\};\label{nzm_hyp}\\
\label{zm_hyp}
\|\lan f\ran(t)\|_{L^2}^2\leq 2\mathfrak{B}\lf(\frac{\kappa }{\nu }+1\rg),\quad\forall t\in[0,T_\star). 
\end{align}
\end{subequations}
As in the works \cite{BGM15I,BedrossianHe16,CotiZelatiDolceFengMazzucato}, to prove Theorem \ref{thm_1} a), it is enough to prove the stronger conclusion on the same time interval:\begin{subequations}
\begin{align}
\|f_\nq(t)\|_{L^2}\leq 2e \|f_{\nq}(0)\|_{L^2}\exp\lf\{-\delta\nu^{1/2}t\rg\};\label{nzm_con}\\
\label{zm_con}
\|\lan f\ran(t)\|_{L^2}^2\leq \mathfrak{B}\lf(\frac{\kappa }{\nu }+1\rg),\quad\forall t\in[0,T_\star).
\end{align}
\end{subequations}
Here to prove the conclusion, we will choose the $\mathfrak{B}$ appropriately and $\kappa,\ \nu$ small enough. {The explicit choices of $\mathfrak B$ is of the form $(\text{constant}+\|\lan f_0\ran\|_{L^2}^2)$. We refer the interested readers to the explicit expression \eqref{chc_B}.} 
 
\noindent
{\bf Step \# 2: Nonlinear enhanced dissipation \eqref{nzm_con}.}
Now we analyze the system \eqref{eq_nq}. To this end, we decompose the time horizon into intervals of length $\delta^{-1}\nu^{-1/2}$, i.e.,
\begin{align*}
[0,\infty)=\bigcup_{i=0}^\infty\ [i\delta^{-1}\nu^{-1/2},(i+1)\delta^{-1}\nu^{-1/2})=:\bigcup_{i=0}^\infty\ [T_i, T_{i+1}).
\end{align*}
Here, $\delta$ is chosen such that 
\begin{align}
C_0\exp\lf \{-\frac{\delta_0}{\delta}\rg\} \leq \frac{1}{32}.\label{chc_delta}
\end{align}
Here, the constants $C_0,\ \delta_0$ are constants defined in Theorem \ref{thm:led}. 
We consider the passive scalar solution initiated from $t=T_i$ with data $f_\nq(T_i, \bx,\theta)$, i.e.,
\begin{align}\label{eta_nq}
\pa_t \eta_\nq +v\bp\cdot \na_\bx \eta_\nq=\nu\pa_{\theta}^2 \eta_\nq,\quad \eta_\nq(t=T_i)=f_\nq(t=T_i). 
\end{align} 
To show that there exists nonlinear enhanced dissipation for the $f_\nq$, we compute the deviation between $f_\nq$ and $\eta_\nq$ on the interval $[T_i,T_{i+1})$. Standard $L^2$-energy estimate yields that
\begin{align} \label{DTR}
\frac{1}{2}\frac{d}{dt}\|f_\nq\|_{L^2}^2=&-\nu \|\pa_\theta  f_\nq\|_{L^2}^2+\kappa \iint \pa_\theta f_\nq  \lf(L[f]_\nq \lan f\ran+\lan L[f]\ran f_\nq+L[f]_\nq f_\nq\rg) d\bx d\theta\\
=:&-\mathfrak{D}+\mathfrak{T}_{\nq 0}+\mathfrak T_{0\nq}+\mathfrak T_{\nq \nq}; \end{align}\begin{align} 
\label{DTD} \quad\frac{1}{2}\frac{d}{dt}\|&f_\nq-\eta_\nq\|_{L^2}^2\\
=&-\nu \|\pa_\theta (f_\nq-\eta_\nq)\|_{L^2}^2+\kappa \iint \pa_\theta( f_\nq-\eta_\nq) \lf(L[f]_\nq \lan f\ran+\lan L[f]\ran f_\nq+L[f]_\nq f_\nq\rg) d\bx d\theta\\
=:&-{D}+T_{\nq 0}+T_{0\nq}+T_{\nq \nq}. 
\end{align} 
We observe that the terms $ \mf T$ and $T$ have similar structures, and hence we focus on one set of them. To understand the nonlinearity, we observe the following fact
\begin{align}
\lan L[f]\ran(\theta)=&\frac{1}{(2\pi)^2}\int_{{\mathbb T}}\Psi(w-\theta)\int_{{\mathbb T}^2}\int_{{\mathbb T}^2} \Phi(\bx-\by)f(\by,w) d\by d\bx dw\\
=&\frac{1}{(2\pi)^2}\int_{{\mathbb T}}\Psi(w-\te)\int_{{\mathbb T}^2}\int_{{\mathbb T}^2} \Phi(\by)f(\bx-\by,w)  d\bx d\by dw\\
=&\int_{{\mathbb T}}\Psi(w-\te)\int_{{\mathbb T}^2} \Phi(\by)\lan f\ran(w) d\by dw\\
=&\sima{\int\Phi d\bx \int_{{\mathbb T} } \lan f\ran(w) \Psi(w-\te)dw},\\
 L[\lan f \ran](\bx,\theta)=&\int_{{\mathbb T}}\Psi(w-\te)\int_{{\mathbb T}^2} \Phi(\bx-\by)\lan f\ran(w) d\by dw\\
 =&\sima{\int \Phi d\bx \int_{{\mathbb T}}\Psi(w-\te)\lan f\ran(w)dw}.\label{exp_L<f>}
\end{align}
Hence,
\begin{align}\label{rel}
\lan L[f]\ran=L[\lan f\ran], \quad L[f]_\nq=L[f]-\lan L[f]\ran=L[f]- L[\lan f\ran]=L[f-\lan f\ran]=L[f_\nq].
\end{align}
Moreover, since $\Phi\in C^\infty({\mathbb T}^2),\ \Psi\in C^\infty({\mathbb T})$, we have that 
\begin{align}\label{ker_est}
\|L[f]\|_{L^\infty}=&\max_{(\bx,\te)\in {\mathbb T}^2\times{\mathbb T}}\lf|\int_{{\mathbb T}^2\times{\mathbb T}} \Phi(\bx-\by)\Psi(w-\te)f(\by,w) d\by dw\rg|\\
\leq &C\min\{ \|\Phi\Psi\|_{L_{\bx,\te}^\infty}\|f\|_{L_{\bx,\te}^1}, \|\Phi\Psi\|_{L^2_{\bx,\te}}\|f\|_{L_{\bx,\te}^2}\}. 
\end{align}
Now we estimate the terms in \eqref{DTD}. To begin with, we consider the the $T_{\nq0}$ term, and estimate it with the bound \eqref{ker_est} and the zero mode bound \eqref{zm_hyp},
\begin{align}\label{Tn0}
|T_{\nq 0}|\leq &\frac{1}{4}\nu\|\pa_\theta (f_\nq-\eta_\nq)\|_{L^2}^2+\frac{\kappa^2}{\nu}\|L[f_\nq]\|_{L^\infty}^2\|\lan f\ran\|_{L^2}^2\\
\leq&\frac{1}{4}\nu\|\pa_\theta (f_\nq-\eta_\nq)\|_{L^2}^2+C\frac{\kappa^2}{\nu}\|\Phi\Psi\|_{L^2}^2\|f_\nq \|_{L^2}^2\mathfrak{B}\lf(1+\frac{\kappa}{\nu}\rg).
\end{align}
Next, we estimate the $T_{0\nq}$ term in \eqref{DTD}. To this end, we invoke the bound \eqref{ker_est}, and the conservation of mass \eqref{cons_mass} to obtain
\begin{align*}
|T_{0\nq}|\leq& \frac{1}{4}\nu\|\pa_\theta(f_\nq-\eta_\nq)\|_{L^2}^2+\frac{\kappa^2}{\nu}\|L[\lan f\ran]\|_{L^\infty}^2\|f_\nq\|_{L^2}^2\\
\leq& \frac{1}{4}\nu\|\pa_\theta(f_\nq-\eta_\nq)\|_{L^2}^2+C\frac{\kappa^2}{\nu}\|\Phi\Psi\|_{L^\infty}^2\|\lan f\ran\|_{L^1}^2\|f_\nq\|_{L^2}^2\\
\leq & \frac{1}{4}\nu\|\pa_\theta(f_\nq-\eta_\nq)\|_{L^2}^2+C\frac{\kappa^2}{\nu}\|\Phi\Psi\|_{L^\infty}^2\|f_0\|_{L^1}^2\|f_\nq\|_{L^2}^2.
\end{align*} 
Finally, combining the estimate \eqref{ker_est} and the bound $\|f_\nq\|_{L^1}\leq 2\|f\|_{L^1}=2$ \eqref{Mass}, the $T_{\nq\nq}$ term can be estimated as follows
\begin{align*}
|T_{\nq\nq}|\leq& \frac{1}{4}\nu\|\pa_\theta(f_\nq-\eta_\nq)\|_{L^2}^2+\frac{\kappa^2}{\nu}\|L[ f_\nq]\|_{L^\infty}^2\|f_\nq\|_{L^2}^2\\
\leq &\frac{1}{4}\nu\|\pa_\theta(f_\nq-\eta_\nq)\|_{L^2}^2+C\frac{\kappa^2}{\nu}\|\Phi\Psi\|_{L^\infty}^2\|f_\nq\|_{L^1}^2\|f_\nq\|_{L^2}^2\\
\leq& \frac{1}{4}\nu\|\pa_\theta(f_\nq-\eta_\nq)\|_{L^2}^2+\frac{\kappa^2}{\nu}C(\|\Phi\Psi\|_{L^\infty})\|f_\nq\|_{L^2}^2.
\end{align*} 
Combining the estimates above, we have that
\begin{align}
\frac{d}{dt}\|f_\nq-\eta_\nq\|_{L^2}^2\leq -\frac{1}{2}\nu\|\pa_\theta(f_\nq-\eta_\nq)\|_{L^2}^2+C(\|\Phi\Psi\|_{L^2\cap L^\infty})\frac{\kappa^2}{\nu}\lf(\frac{\kappa}{\nu}\mathfrak B+1\rg)\|f_\nq\|_{L^2}^2.
\end{align}
Through similar estimates on the equation \eqref{DTR} (one can replace $\pa_\te (f_{\nq}-\eta_\nq)$ by $\pa_\te f_\nq$ and run a similar argument), we have 
\begin{align}
\frac{d}{dt}\|f_\nq\|_{L^2}^2\leq C(\|\Phi\Psi\|_{L^2\cap L^\infty})\frac{\kappa^2}{\nu}\lf(\frac{\kappa}{\nu}\mathfrak B+1\rg)\|f_\nq\|_{L^2}^2 .
\end{align}
As a consequence of the above differential inequalities, we have that by the Gr\"onwal inequality, for all $t\in[T_i,T_{i+1}=T_i+\delta^{-1}\nu^{-1/2}]$,
\begin{align}
\label{diff_f_eta}\|f_{\nq}\|_{L^2}^2(t)\leq& \|f_{\nq}(T_i)\|_{L^2}^2\exp\lf\{   \frac{\kappa^2}{\delta\nu^{3/2}}C(\Phi,\Psi)\lf(\mf B\frac{\kappa}{\nu}+1\rg) \rg\};\\
\|f_\nq-\eta_\nq\|_{L^2}^2(t)\leq& C(\Phi,\Psi)\frac{\kappa^2}{\delta\nu^{3/2}}\lf(\mf B\frac{\kappa}{\nu}+1\rg)\|f_{\nq}(T_i)\|_{L^2}^2\exp\lf\{   \frac{\kappa^2}{\delta\nu^{3/2}}C(\Phi,\Psi)\lf(\mf B\frac{\kappa}{\nu}+1\rg) \rg\}.
\end{align}
Note that we are focusing on an interval of length $\delta^{-1}\nu^{-1/2}$, and are assuming that $\kappa\leq \nu^{5/6+\wt\gamma}\leq 1,\, \wt\gamma>0$ \eqref{nl_ED}. Hence, we can choose $\nu\leq \nu_0(\delta,\mathfrak{B}, \|\Phi\Psi\|_{L^2\cap L^\infty}) )$ small enough such that 
\begin{align}\label{chc_nu_1}
\|f_{\nq}(T_i+\tau)\|_{L^2}^2\leq& 2\|f_{\nq}(T_i)\|_{L^2}^2,\\
\|f_\nq-\eta_\nq\|_{L^2}(T_i+\tau)\leq& \frac{1}{4}\|f_{\nq}(T_i)\|_{L^2},\quad\forall \tau\in[0,\delta^{-1}\nu^{-1/2}].
\end{align}
Thanks to the  choice of $\delta$ \eqref{chc_delta}, we have that 
\begin{align*}
\|\eta_\nq(T_i+\delta^{-1}\nu^{-1/2})\|_{L^2}\leq \frac{1}{32}\|f_{\nq}(T_i)\|_{L^2}.
\end{align*}
Combining the estimate above, we have that 
\begin{align}
\|f_\nq(T_{i+1})\|_{L^2}\leq 
\|f_\nq(T_{i+1})-\eta_\nq(T_{i+1})\|_{L^2}+\|\eta_\nq(T_{i+1})\|_{L^2}\leq \frac{1}{e}\|f_\nq(T_i)\|_{L^2}.\label{disc_ED}
\end{align}
The remaining argument to derive the enhanced dissipation estimate \eqref{nzm_con} is standard. If $t\in \delta^{-1}\nu^{-1/2}\mathbb{N}$, then we have that by \eqref{disc_ED},
\begin{align}
\|f_\nq(t)\|_{L^2}\leq \|f_\nq(0)\|_{L^2}\exp\lf\{- \frac{ t}{\delta^{-1}\nu^{-1/2}}\rg\}=\|f_{0;\nq}\|_{L^2}\exp\lf\{- {\delta \nu^{ 1/2}}{ t}\rg\}.
\end{align} On the other hand, if $t\notin\delta^{-1}\nu^{-1/2}\mathbb{N}$, we choose the largest integer $N$ such that $\delta^{-1}\nu^{-1/2}N\leq t.$ Hence, we have the relation $t\in\delta^{-1}\nu^{-1/2}[N,N+1]$. Then, we combine the estimates \eqref{chc_nu_1} and \eqref{disc_ED} to obtain that
\begin{align}
\|f_\nq(t)\|_{L^2}\leq& 2\|f_\nq (T_N)\|_{L^2}\leq 2e\|f_{0;\nq}\|_{L^2}e^{-(N+1)}\leq 2e\|f_{0;\nq}\|_{L^2}e^{ -\delta \nu^{1/2}t },\quad \forall  t \in[0,\infty).
\end{align} 
Hence we have proven the nonlinear enhanced dissipation \eqref{nzm_con}.  

\noindent
{\bf Step \# 3: The $\bx$-average estimate.}
Now we consider the zero-mode estimate $\lan f\ran.$ The equation can be rephrased as follows:
\begin{align*}
\pa_t \lan f\ran +\kappa\pa_\theta( L[\lan f\ran] \lan f\ran) + \kappa\pa_\theta \lan L[ f_\nq] f_\nq\ran =\nu\pa_{\theta}^2 \lan f\ran. 
\end{align*}
We apply the $L^2$-energy estimate to obtain that 
\begin{align*}
\frac{1}{2}\frac{d}{dt}\|\lan f\ran\|_{L^2}^2=&-\nu\|\pa_\theta\lan f\ran\|_{L^2}^2+\frac{\kappa^2}{\nu}\|L[f_\nq]\|_{L^\infty}^2\|f_\nq\|_{L^2}^2+\frac{\kappa^2}{\nu}\|L[\lan f\ran]\|_{L^\infty}^2\|\lan f\ran\|_{L^2}^2.
\end{align*}
Thanks to the enhanced dissipation estimate \eqref{nzm_hyp} and the $L$-bound \eqref{ker_est}, we have that 
\begin{align*}
\frac{\kappa^2}{\nu}\|L[f_\nq]\|_{L^\infty}^2\|f_\nq\|_{L^2}^2\leq C(\|\Phi\Psi\|_{L^\infty})\frac{\kappa^2}{\nu}\|f_\nq(0)\|_{L^2}^2\exp\{-2\delta\nu^{1/2}t\}.  
\end{align*} 
We further observe that the time integral of the above quantity is bounded as follows
\begin{align}
G(t):=\int_0^t C\frac{\kappa^2}{\nu}\|f_\nq(0)\|_{L^2}^2\exp\lf\{-2\delta\nu^{1/2}\tau\rg \} d\tau\leq \frac{C\kappa^2}{\delta\nu^{3/2}}\|f_\nq(0)\|_{L^2}^2\leq \frac{\kappa }{\nu }(C\delta^{-1}\nu^{1/3+\wt{\gamma}}\|f_\nq(0)\|_{L^2}^2).
\end{align}
Here we have invoked the choice of $\kappa$, i.e., $\kappa\leq \nu^{5/6+\wt{\gamma}},\ \wt{\gamma}>0$. 
Hence we can choose $\nu$ small enough such that\begin{align}\label{chc_nu1}
C\delta^{-1}\nu^{\wt{\gamma} +1/3}\|f_\nq(0)\|_{L^2}^2\leq 1\quad\Rightarrow\quad G(t)\leq \kappa\nu^{-1}.
\end{align}
Now we have that 
\begin{align}
\frac{d}{dt}(\|\lan f\ran\|_{L^2}^2-G(t))\leq &-\nu\|\pa_\theta\lan f\ran\|_{L^2}^2+\frac{\kappa^2}{\nu}\|L[\lan f\ran]\|_{L^\infty}^2\|\lan f\ran\|_{L^2}^2\\
\leq& -\nu\|\pa_\theta\lan f\ran\|_{L^2}^2+\frac{\kappa^2}{\nu}C(\|\Phi\Psi\|_{L^\infty})\|\lan f\ran\|_{L^2}^2.
\end{align}
Now we recall the Nash inequality on the torus ($\|f\|_{L^1}\neq 0$)
\begin{align}\label{Nash}
&\|\lan f \ran\|_{L^2}\leq C\|\lan f\ran\|_{L^1}^{2/3}\|\pa_\theta \lan f\ran\|_{L^2}^{1/3}+C\|\lan f\ran\|_{L^1} \ \Rightarrow \
\|\lan f \ran\|_{L^2}^6\leq  C\|\lan f\ran\|_{L^1}^{ 4}\|\pa_\theta \lan f\ran\|_{L^2}^{2}+C\|\lan f\ran\|_{L^1}^6\\
&\Rightarrow
-\|\pa_\theta \lan f\ran\|_{L^2}^2\leq -\frac{\|\lan f \ran\|_{L^2}^6}{C\|\lan f\ran\|_{L^1}^4}+C\|\lan f\ran\|_{L^1}^2.
\end{align}
Hence, there is the following relation
\begin{align*}
\frac{d}{dt}(\|\lan f\ran\|_{L^2}^2-G(t) )\leq -\nu\frac{\|\lan f\ran\|_{L^2}^6}{\|\lan f\ran\|_{L^1}^4}+\nu\|\lan f\ran\|_{L^1}^2+\frac{C\kappa^2}{\nu}\|\lan f\ran\|_{L^2}^2.
\end{align*} 
\sima{Since we only care about the upper bound of the $\|\lan f\ran\|_{L^2}^2$, we focus on the time intervals on which $\|\lan f\ran(t)\|_{L^2_\te}^2\geq 2\|G(\cdot)\|_{L^\infty([0,T_\ast))}$. Let $\mf I=(a,b)\cap[0,T_\ast)$ be an arbitrary (non-extendable) time interval such that $\|\lan f\ran(t)\|_{L^2_\te}^2> 2\|G(t)\|_{L^\infty([0,T_\ast))}$. The interval is non-extendable in the sense that at the left end point $t=a$, either $\|\lan f\ran(a)\|_{L^2_\te}=2\|G(\cdot)\|_{L_t^\infty([0,T_\ast))}$ or $a=0$, and at the right end point $t=b$, either $\|\lan f\ran(b)\|_{L^2_\te}=2\|G (\cdot)\|_{L_t^\infty([0,T_\ast))}$ or $b=T_\ast$. On the interval $\mf I$, we consider the quantity $\mathcal{Z}(t):=\|\lan f(t)\ran\|_{L^2}^2-G(t)$ and rewrite the above relation in the following fashion using the fact that $\|\lan f\ran\|_{L_\te^1}=\frac{1}{2\pi}\|f_0\|_{L^1_{\bx,\te}}=\frac{1}{2\pi}$:
\begin{align*}
&\frac{d}{dt}\mathcal{Z}\leq -\nu\lf((2\pi)^4 (\|\lan f\ran\|_{L^2}^2-G+G)^3  -1-\frac{C\kappa^2}{\nu^2}\lf(\|\lan f\ran\|_{L^2}^2-G\rg)-\frac{C\kappa^2}{\nu^2}G\rg)\\
&\leq -\nu\Bigl(\mathcal{Z}^3-C\kappa^2\nu^{-2}\mathcal{Z}-C G^{3}-1-C\kappa^2\nu^{-2}G\Bigr).
\end{align*} 
We observe that for $\mathcal{Z}=\|\lan f\ran\|_{L^2}^2-G$ large, the cubic term dominates the others. Hence, one obtain the following bound
\begin{align}
&\|\lan f\ran(t)\|_{L^2}^2=\mathcal{Z}(t)+G(t)\\
&\leq \|\lan f\ran(a)\|_{L^2}^2+1+C\kappa\nu^{-1}+C\|G (\cdot)\|_{L_t^\infty([0,T_\ast))}+C\kappa^{2/3}\nu^{-2/3}\|G (\cdot)\|_{L_t^\infty([0,T_\ast))}^{1/3}\\
 &\leq  \|\lan f_0\ran\|_{L^2}^2+1+C\kappa\nu^{-1},\quad \forall t\in \mf I .  
\end{align}
For the time where $\|\lan f\ran(t)\|_{L^2_\te}^2\leq 2\|G\|_{L_t^\infty([0,T_\ast))}$, the above estimate also holds. As a result, 
}
\begin{align}
\|\lan f\ran\|_{L^2}^2\leq 1+\frac{C\kappa }{\nu }+\|\lan f_0\ran\|_{L^2}^2=(2+C+\|\lan f_0\ran\|_{L^2}^2)\lf(1+\frac{\kappa }{\nu }\rg).\label{fz_est}
\end{align}
Now choose 
\begin{align}\label{chc_B}
\mathfrak{B}=2+C+\|\lan f_0\ran\|_{L^2}^2, 
\end{align}
and we have  
\eqref{zm_con}.

\ifx
\noindent
\myb{{\bf   Step \# 4: Long time estimate.}
Since $f_\nq$ undergoes enhanced dissipation, there exists $T_\ep=T(\ep,\nu,\|f_{0;\nq}\|_{L^2})$ such that
\begin{align*}
\|f_\nq(T_\ep)\|_{L^2}\leq \nu^{1/2}\ep,\quad \forall t\geq T_\ep.
\end{align*}
Now we consider the difference between the nonlinear system \eqref{eq_average} and the effective equation \eqref{effctv_dym} initiated from $\lan f\ran(T_\ep)$, 
\begin{align}
\frac{1}{2}\frac{d}{dt}\| \fz-g \|_{L^2}^2 =& \int \pa_\theta(\fz-g)\lf[\fz L\fz -gLg - \nu\pa_{\theta}(\fz-g)\rg] d\theta\\
&+ \int \pa_\theta(\fz-g)\kappa f_\nq Lf_\nq d\theta\\
\leq&- \frac{\nu}{2}\|\pa_\theta (\fz-g)\|_{L^2}^2+ C\frac{\kappa^2}{\nu} \|L[\fz-g]\|_{L^\infty}^2\|f\|_{L^2}^2\\
&+C\frac{\kappa^2}{\nu}\|\fz-g\|_{L^2}^2\|L[g]\|_{L^\infty}^2+C\|f_\nq\|_{L^2}^2\|Lf_\nq\|_{L^\infty}^2\\
\leq & - \frac{\nu}{2}\|\pa_\theta (\fz-g)\|_{L^2}^2+ C\frac{\kappa^2}{\nu} \| \fz-g \|_{L^2}^2\mf B\frac{\kappa}{\nu}+C\frac{\kappa^2}{\nu}\|\fz-g\|_{L^2}^2+C\|f_\nq\|_{L^2}^2.
\end{align}
Here the implicit constant depends on $\|\Psi\Phi\|_{L^2\cap L^\infty}$. 
Hence, by the Gr\"onwall inequality, we have that  
\begin{align}
\|\fz-g\|_{L^2}^2(t)\leq \int_{T_\ep}^t \exp\lf\{C\mf B\frac{\kappa^3}{\nu^2}(t-s)\rg\}\|f_{\nq}(T_\ep)\|_{L^2}^2\exp\lf\{-\delta\nu^{1/2}(s-T_\ep)\rg\}ds.
\end{align}
Now we recall that $\kappa\leq \nu^{5/6+\wt{\gamma}},\ \wt{\gamma}>0$ and $\nu$ can be chosen to be small. Hence we have that
 \begin{align}
 \|\fz-g\|_{L^2}^2(t)\leq \frac{C\|f_{\nq}(T_\ep)\|_{L^2}^2}{\delta\nu^{1/2}}\lf(\exp\{\nu^{1/2}(t-T_\ep)\}-\exp\{-\delta\nu^{1/2}(t-T_\ep)\}\rg)\leq \nu,
 \end{align}
 as long as $\nu$ is small enough, and $t-T_\ep\leq \nu^{-1/2}.$}
\fi
\end{proof}
 
\begin{proof}[Proof of Theorem \ref{thm_1} {\bf b)}]
Now we focus on the initial time layer $t\in[0,\delta^{-1}\nu^{-1/2} ]$ to derive the nonlinear mixing estimate \eqref{nonlin_ID}. Consider the solution $f_\nq$ of the actual equation \eqref{eq:bsc_1} and the passive scalar solution $\eta_\nq$ initiated from the same initial data. Thanks to the argument in the previous proof (e.g. \eqref{diff_f_eta}) and the definition of $\mf B$ \eqref{chc_B}, we have that   
\begin{align}
\|f_\nq-\eta_\nq\|_{L^2}\leq C(\Phi,\Psi,\delta^{-1},\|\lan f_0\ran\|_{L^2})\lf(\frac{\kappa^{3/2} }{\delta^{1/2} \nu^{5/4}}+\frac{\kappa }{\delta^{1/2} \nu^{3/4}}\rg)\|f_\nq(0)\|_{L^2},\quad t\in[0,\delta^{-1}\nu^{-1/2}].
\end{align}
\sima{Combining the linear mixing estimate of the passive scalar solution $\eta_\nq$ \eqref{linear_ID}, the parameter constraint $\kappa\leq C_\dagger\nu$, and the deviation estimate
\begin{align}
    \|f_\nq-\eta_\nq\|_{\dot H^{-1}}\leq \|f_\nq-\eta_\nq\|_{L^2},
\end{align}
the following relation holds for $t\leq \delta^{-1}\nu^{-1/2} $:}
\begin{align}
    &\|f_\nq(t)\|_{\dot H^{-1}}\leq \|f_\nq-\eta_\nq\|_{\dot H^{-1}}(t)+\|\eta_\nq(t)\|_{\dot H^{-1}}\leq\\
    &  C(\Psi, \Phi,\delta^{-1}, C_\dagger,\|\lan f_0\ran\|_{L^2})\lf( \nu^{1/4} \|f_\nq(0)\|_{L^2}+\min\lf\{\frac{\nu^{1/4}}{\min\{1,\nu^{1/4}t^{1/2}\}},e^{-\delta_0\nu^{1/2}t}\rg\}\|f_\nq(0)\|_{H^1}\rg).
\end{align}Hence, we obtain that for all 
 $t\leq \delta^{-1}\nu^{-1/2}$, 
\begin{align}
    \|f_\nq(t)\|_{\dot H^{-1}}\leq C(\Psi, \Phi,\delta, C_\dagger,\|\lan f_0\ran\|_{L^2})\frac{1}{t^{1/2}} \|f_\nq(0)\|_{H^1}.
\end{align} \end{proof}

\begin{remark}
    As a corollary of the above mixing estimate, we have the following estimate about the strength of the interaction for all $t\leq \delta^{-1}\nu^{-1/2}$:
\begin{align}
    &\|L[f_\nq]\|_{L^\infty}=\sup_{\bx,\te}\lf|L[f_\nq](t,\bx,\te)\rg|=\lf|\iint \Phi(\bx-\by)\Psi(\te-w) f_\nq(\by, w) d\by dw \rg|\\ & \leq \|f_\nq\|_{\dot H^{-1}}\|\Phi\Psi-\overline{\Phi\Psi}\|_{H^1}\leq C\frac{1}{ t^{1/2} }\|f_\nq(0)\|_{H^1}.
\end{align}
Here, the last constant $C=C(\Psi, \Phi,\delta^{-1}, C_\dagger,\|\lan f_0\ran\|_{L^2})$.
\end{remark}



\section{Analysis of the Effective Dynamics}\label{sec:eff_sym}
In this section, we prove Theorem \ref{thm_2} and \ref{thm_3}. 

\begin{proof}[Proof of Theorem \ref{thm_2} ] Thanks to the relation \eqref{U} and the constraint \eqref{phi_sym}, the equation \eqref{effctv_dym} can be reformulated as 
\begin{align}
\pa_t g-\kappa\pa_\theta (g (\pa_\theta \mathbb{U}*g))=\nu\pa_{\theta}^2 g.
\end{align}

\noindent
{\bf Step \# 1: Proof of statement a).} 
We compute the time derivative of $F[g]$
\begin{align}
\frac{d}{dt}F[g]=&\nu\int g_t \log g d\theta + \kappa \iint \mathbb{U}(\theta-w)g_t (\theta)g(w)d\theta dw\\
=&\int \pa_\theta(\nu g\pa_\theta\log g +\kappa g\pa_\theta \mathbb{U}*g ) (\nu\log g +\kappa \mathbb{U}*g)d\theta  \\
=&-\int  g(\nu\pa_\theta\log g +\kappa \pa_\theta \mathbb{U}*g ) (\nu\pa_\theta\log g +\kappa \pa_\theta \mathbb{U}*g)d\theta\\
=&-\int g|\nu\pa_\theta\log g +\kappa \pa_\theta \mathbb{U}*g |^2d\theta. 
\end{align} 
This is \eqref{Fisher}. 
\myh{$\frac{d}{dt}\mathcal{D}[g]?$ Maybe there are no clean way to show that it is actually decreasing. Moreover, in the full problem, there will be error terms. }


\noindent
{\bf Step \# 2: Proof of statement b).} 
We can do a simple linear analysis to identify the instance where the phase transition happens. First of all, we observe that $\overline{g}=\frac{1}{2\pi}$ is a solution and linearize the problem around this stationary solution: ($g=\wt g+\overline{g}$)
\begin{align}
\pa_t \wt g - \frac{\kappa}{2\pi}  \pa_\theta (\pa_\theta \mathbb{U}* \wt g) =\nu\pa_{\theta}^2\wt g.
\end{align}
Now we take the Fourier transform to obtain
\begin{align}
\pa_t \wt g_\ell+ \frac{\kappa}{2\pi}    \wh {\mathbb{U}}(\ell)|\ell|^2\wt g_\ell=&-\nu|\ell|^2\wt g_\ell.
\end{align}
If 
\begin{align}
\frac{\kappa}{2\pi}  | \wh {\mathbb{U}}(\ell)|< \nu,\quad \forall \ell\neq0,
\end{align}
then the constant state is linearly stable.  We would like to highlight that, in the $\Psi_0(\cdot)=\sin(\cdot)$ case, the relation is simple and has the form 
\begin{align}\label{lin_r}
\frac{\kappa}{2}<\nu. 
\end{align} If there exists $\ell\in \mathbb{Z}\backslash\{0\}$ such that
\begin{align}
-\frac{\kappa}{2\pi}\Re\wh{\mathbb{U}}(\ell)-\nu>0,
\end{align} then the $\ell$ is an unstable growing mode. 

This concludes the proof of Theorem \ref{thm_2}.

\ifx 
We assume that 
\begin{align}
\frac{\kappa}{2\pi}\int \Phi d\bx<\nu.
\end{align}
Then the nonzero mode contributions are exponentially decreasing in time. The system will eventually converge to the stationary state $\frac{M}{2\pi}$. \fi

\end{proof}




Our proof of the first part of  Theorem \ref{thm_3} is in the same spirit as the proof of Proposition 3.2 in \cite{FrouvelleLiu12} but with major adjustments to keep track of the dynamics of the remainder $f_\nq. $
\begin{proof}[Proof of Theorem \ref{thm_3}, {\bf Part a)}]
Since the constant state $g=\text{const}.$ is always a solution to the stationary equation \begin{align}\mathcal{D}[g]=\int g|\nu\pa_\theta\log g+\kappa \mathbb{U}*\pa_\theta g|^2d\theta=0,
\end{align} the set $ {\mathcal{S}_{\infty}}\neq \emptyset.$  We recall the equation \eqref{eq_average}, and compute the time evolution of the free energy $F[\fz]$ \eqref{Free_energy},
 \begin{align}\label{DTL} \\
&\hspace{-0.25cm}\frac{d}{dt}F[\fz]\\
=&\nu\int \pa_t\fz \log \fz d\theta + \kappa \iint \mathbb{U}(\theta-w)\pa_t\fz (\theta)\fz(w)d\theta dw\\
=&\int \pa_\theta(\nu \fz\pa_\theta\log \fz +\kappa \fz\pa_\theta \mathbb{U}*\fz -\kappa \lan L[f_\nq]f_\nq\ran) (\nu\log \fz +\kappa \mathbb{U}*\fz)d\theta  \\
=&-\int \fz|\nu\pa_\theta\log \fz +\kappa \pa_\theta \mathbb{U}*\fz |^2+\kappa\int \lan L[f_\nq]f_\nq\ran\lf(\nu\frac{\pa_\theta\lan f\rangle}{\lan f\ran}+\kappa \pa_\theta\mathbb{U}*\lan f\ran\rg)\\
=&\sima{-\int \fz|\nu\pa_\theta\log \fz +\kappa \pa_\theta \mathbb{U}*\fz |^2+\kappa\nu\int\lf \lan\frac{ L[f_\nq]f_\nq}{\lan f\ran}\rg\ran\pa_\theta\lan f\rangle+\kappa^2\int \lan L[f_\nq]f_\nq\ran \pa_\theta\mathbb{U}*\lan f\ran\n} \\
=:&-\mathcal{D}[\fz]+\mathbb{T}_1+\mathbb{T}_2.
 \end{align}
First, we estimate the $\mathbb{T}_1$-term. We note that $\lan f\ran(\te)+f_\nq(\bx,\theta) =f(\bx,\theta)\geq 0$, so the following estimate holds
\begin{align*}
\lan f\ran(\te)\geq (f_\nq)^-(\bx,\te)\;\Longrightarrow \; 2\pi\lan f\ran(\theta) \geq \frac{1}{2}\|f_\nq(\cdot,\te)\|_{L^1_\bx}\;\Longrightarrow\; \lf\|\lf\lan L[f_\nq]\frac{f_\nq}{\lan f \ran}\rg\ran\rg\|_{L_\theta^\infty}\leq 4\pi\|L[f_\nq]\|_{L_{\bx,\te}^\infty}.
\end{align*}
Here, we use the fact that $f_0>0$ and the $\min_{\bx,\theta} f(t,\bx,\te)>0$ for any finite $t.$ \sima{By the parameter constraint $\kappa\leq\nu^{5/6}\leq 1$, the $L[f_\nq]$ estimate \eqref{ker_est},  the enhanced dissipation \eqref{nl_ED}, and the higher regularity estimate  \eqref{Hn_bnd}, we obtain the following estimate,
\begin{align*}
&|\mathbb{T}_{1}|\leq C\kappa\nu\|\pa_\theta f\|_{L_{\bx,\te}^2}\|L[f_\nq]\|_{L_{\bx,\te}^\infty}&&\qquad (\|\pa_\te \fz\|_{L_{\te}^2}\leq C\|\pa_\theta f\|_{L_{\bx,\te}^2}.)\\ 
& \leq C(\Phi,\Psi)\kappa\nu\frac{\max\{1,\|f_0\|_{H_{\bx,\theta}^1}\}}{\nu^{3/4}}\exp\{C\nu^{2/3}t\}\|f_\nq(t)\|_{L_{\bx,\te}^2} &&\qquad(\text{Apply }\eqref{Hn_bnd} \text{ and }\eqref{ker_est}.)\\
&\leq C\kappa\nu^{1/4}\max\{1,\|f_0\|_{H_{\bx,\theta}^1}\}\|f_{0;\nq}\|_{L_{\bx,\te}^2}\exp\{C\nu^{2/3}t-\delta \nu^{1/2}t\}&&\qquad(\text{Apply }\eqref{nl_ED}.)\\
& \leq C{\nu^{13/12}}(1+\|f_0\|_{H_{\bx,\te}^1}) \|f_{0;\nq}\|_{L_{\bx,\te}^2}\exp\{C\nu^{2/3}t-\delta \nu^{1/2}t\}. &&\qquad(\text{Apply }\kappa\leq  \nu^{5/6}.)
\end{align*}
Next, we estimate the $\mathbb{T}_2$-term using the enhanced dissipation \eqref{nl_ED}, the $\lan f\ran$-estimate \eqref{nl_z_mod}, and the $L[f_\nq]$ estimate \eqref{ker_est} as follows:
 \begin{align}
&| \mathbb{T}_2| \leq C\kappa^2\|L[f_\nq]\|_{L^\infty_{\bx,\te}}\|f_\nq\|_{L_{\bx,\te}^2}\|\pa_\te \mathbb{U}\|_{L_{\bx,\te}^1}\|\lan f\ran\|_{L_{\bx,\te}^2}&&(\text{H\"older and Young inequalities.})\\
&\leq C\kappa^2\|f_\nq\|_{L_{\bx,\te}^1}\; \|f_{0;\nq}\|_{L_{\bx,\te}^2}\exp\{-\delta\nu^{1/2}t\}\|\lan f\ran\|_{L_{\bx,\te}^2} &&(\text{Apply }\eqref{nl_ED},  \eqref{ker_est}.)\\
&\leq C\lf(\kappa^2+\frac{\kappa^{5/2}}{\nu^{1/2}}\rg)(1+\|f_{0}\|_{L_{\bx,\te}^2})\|f_{0}\|_{L_{\bx,\te}^2}\exp\{-\delta\nu^{1/2}t\}.&&(\text{Apply \eqref{nl_z_mod} and }\|f_\nq\|_{L_{\bx,\te}^1}\leq2.)
 \end{align}}
Hence,  we see that the $|\mathbb{T}_{1}|+|\mathbb T_{2}|\rightarrow0$ as $t\rightarrow\infty$. Moreover,
\begin{align}
\lim_{T\rightarrow \infty}\int_T^\infty |\mathbb{T}_1|+|\mathbb{T}_2|dt=0.
\end{align}
\ifx
\myr{Hopefully, we can use the argument in \cite{FrouvelleLiu12} to finish the proof. 
 Formally, we have that
 \begin{align*}
 \lim_{T\rightarrow\infty}\int_T^\infty\int \fz|\nu\pa_\te \log \fz+\kappa \Psi*\fz|^2 d\te dt=\lim_{T\rightarrow\infty}\int_{T}^\infty \mathcal{D}[\fz]dt=0.
 \end{align*}
 Moreover, the $\mathcal{D}[\fz]=0$ only if $\fz\in \mathcal{E}_\infty$. Hence the ``distance'' $dist(\fz,\mathcal{E}_\infty)\sim\mathcal D[\fz]\rightarrow 0. $}
\fi  As a consequence of \eqref{f_nq_Hn}, \eqref{fz_Hn} and the equation \eqref{eq_average}, we have that $\|\pa_t \lan f\ran\|_{H^n}$ is uniformly  bounded in time. Since all the qualitative requirements in the proof of Proposition 3.2 of \cite{FrouvelleLiu12} are fulfilled, one can follow their argument to get the result directly. We will summarize the main argument in \cite{FrouvelleLiu12}, highlight the main adjustments, and omit further details for the sake of brevity. 
 

\myh{\bf BLUE: The argument in the paper with our modifications. This part should NOT appear in the actual paper unless significant rewriting happens.}

First of all, we recall that the initial data $f_0$ is $C^\infty$, and all $H^M$-norms are propagated \eqref{f_nq_Hn}, \eqref{fz_Hn} ($M\in \mathbb{N}$). Following the paper, we can choose a sequence of distinct times $\{t_n\}_{n=1}^\infty$ such that $\lim_{n\rightarrow \infty}t_n={\infty}$ and $\lim_{n\rightarrow\infty} f(t_n)=f_\infty$ in the $L^2$ sense. Since the sequence $\{f(t_n)\}_{n=1}^\infty$ is also uniformly bounded in arbitrary $H^M$-space, the Gagliardo-Nirenberg interpolation yields that the limiting function $f_\infty$ is in arbitrary $H^M$-space, and hence it is a $C^\infty$ function. 

\myh{\cite{FrouvelleLiu12}: Let $(t_n)$ be an unbounded increasing sequence of time, and suppose that $\{\lan f\ran(t_n)\}_{n=1}^\infty$ converges in $H^s({\mathbb T})$ to $f_{\infty}$ for some $s \in \mathbb{R}$. We first remark that $\lan f\ran(t_n)$ is uniformly bounded in $H^{s+2p}({\mathbb T})$ (by \eqref{fz_Hn}), and then by a simple interpolation estimate, we get that $\|\lan f\ran(t_n) - \lan f\ran(t_m)\|_{\dot H^{s+p}}^2 \leq \|\lan f\ran(t_n) - \lan f\ran(t_m)\|_{\dot H^s} \|\lan f\ran(t_n) - \lan f\ran(t_m)\|_{\dot H^{s+2p}}$, and $\{\lan f\ran(t_n)\}_{n=1}^\infty$ also converges in $H^{s+p}({\mathbb T})$. So $f_{\infty}$ is in any $H^{s}({\mathbb T})$.}

Next, we will show that $\mathcal{D}[f_{\infty}] = 0$. We summarize the argument in \cite{FrouvelleLiu12} as follows. Suppose this is not the case, i.e., $\mathcal{D}[f_{\infty}] >0.$ Thanks to a detailed analysis of the Fisher information $\mathcal{D}$, one is able to show that if $\mathcal{D}[f_\infty]>0$, then there exists a threshold $\delta_Z$ and a constant $Z>0$ such that if $\|\lan f\ran-f_\infty\|_{H^M}\leq \delta_Z$, then $\mathcal{D}[\lan f\ran]\geq Z>0.$

Then we observe that by the uniform-in-time Sobolev bounds \eqref{f_nq_Hn}, \eqref{fz_Hn},  and the equation \eqref{eq_average}, the time derivative $\partial_t \lan f\ran$ is uniformly bounded  in $H^M$. Hence, there exists $\eta > 0$, which depends on $\nu$ and the initial data $f_0$, such that if $t - s \leq \eta$, then $\|\lan f\ran(t) - \lan f\ran(s)\|_{H^M} \leq \frac{\delta_Z}{3}$. We then take $N$ sufficiently large such that $\|\lan f\ran(t_n) - f_{\infty}\|_{H^M} \leq \frac{\delta_Z}{3}$ for all $n \geq N$. Thanks to the continuity argument before, for $n \geq N$, the Fisher information has a lower bound:
\begin{align*}\mathcal{D}[\lan f\ran ] \geq {Z>0}, \quad \forall t\in[t_n, t_n + \eta].
\end{align*} 
Without loss of generality, we assume that $t_{n+1} \geq t_n + \eta$. Hence, by the relation \eqref{DTL}, 
\begin{align*}
F[\lan f(t_n)\ran] - F[\lan f(t_{n+P})\ran] \geq& \int_{t_n}^{t_{n+P}} \mathcal{D}[\lan f\ran]dt-\int_{t_n}^{t_{n+P}}|\mathbb{T}_1|+|\mathbb{T}_2|dt \\
\geq& (P-1)\eta Z-C\delta^{-1}\nu^{-5/12}(1+\|f_0\|_{H^1}^2)\exp\lf\{-\frac{1}{2}\nu^{1/2}T_\ast\rg\}.
\end{align*}
Now, if we choose the $T_\ast(=2\nu^{-1/2}\log[C \delta^{-1}\nu^{-5/12}(1+\|f_0\|_{H^1}^2)]+2\nu^{-1/2}\log(\eta Z/2)^{-1})$ to be large enough such that the last term is dominated by $\frac{1}{2}\eta Z$, it is guaranteed that
 \begin{align*}
F[\lan f(t_n)\ran] - F[\lan f(t_{n+P})\ran] \geq\frac{1}{2}(P-1)\eta Z,
\end{align*}
and the difference in free energy will grow indefinitely. Since the left-hand side is bounded above by $F[\lan f\ran(0)]+\int_0^\infty |\mathbb{T}_1|+|\mathbb{T}_2|dt<\infty$, taking $P$ sufficiently large gives the contradiction. To conclude, we have that $\mathcal{D}[f_\infty]=0.$

{Finally, we prove \eqref{convergence}. Suppose that there exists an increasing subsequence $\{t_n\}_{n=1}^\infty$ with $ \lim_{n\rightarrow\infty}t_{n}=\infty$, such that there exists $M_*\in \mathbb{N}$,
\begin{align}
    \lim_{n\rightarrow \infty}\inf_{G\in \mathcal S_\infty}\|f(t_n,\cdot)-G(\cdot)\|_{H_{\bx,\te}^{M_*}}>0.
\end{align}
 Now we note that this sequence $\{f(t_{n})\}$ is also bounded in $H^{M_*+1}$ by \eqref{f_nq_Hn}, \eqref{fz_Hn}. Hence, we apply compact Sobolev embedding to extract a subsequence $\{t_{n_k}\}_{k=1}^\infty$ such that $f(t_{n_k},\cdot)\xrightarrow{H^{M_*}_{\bx,\te}} f_\infty(\cdot)$. Now, thanks to the Gagliardo-Nirenbergy inequality and the $H^M\, (M\in \mathbb{N})$ bounds \eqref{f_nq_Hn}, \eqref{fz_Hn}, we have that $f(t_{n_k},\cdot)$ converges to $f_\infty$ in other $H^M, \, M> M_*$ spaces. Hence $f_\infty\in C^\infty.$ Moreover, thanks to the argument above, we have that $\mathcal{D}[f_\infty]=0.$ However, this implies that $f_\infty \in \mathcal S_\infty$ and $\lim_{k\rightarrow \infty}\|f(t_{n_k})-f_\infty\|_{H^{M_*}}=0$, which is a contradiction. This concludes the proof. }
\end{proof}

\begin{proof}[Proof of Theorem \ref{thm_3}, {\bf Part b)}]
We note that $\overline{\lan f\ran}=\frac{1}{2\pi}$. Now we compute the time evolution of the quantity $\|\lan f\ran -\overline{\lan f\ran}\|_{L^2}^2$
\begin{align*}
\frac{1}{2}\frac{d}{dt}\lf\|\lan f\ran -\overline{\lan f\ran}\rg\|_{L^2}^2=& -\nu\lf\|\pa_\theta\lf(\lan f\ran -\overline{\lan f\ran}\rg)\rg\|_{L^2}^2+\kappa \int \pa_\theta\lf(\lan f\ran -\overline{\lan f\ran}\rg) \lf\lan f L[f]\rg\ran d\theta\\
=& -\frac{\nu}{2}\lf\|\pa_\theta\lf(\lan f\ran -\overline{\lan f\ran}\rg)\rg\|_{L^2}^2+\kappa \int \pa_\theta\lf(\lan f\ran -\overline{\lan f\ran}\rg)  \lf(\lan f\ran L[\lan f\ran]+\lan f_\nq L[f_\nq]\ran\rg) d\theta. \quad
\end{align*}
We apply the H\"older inequality, the Poincar\'e inequality, the relation $L[\lan f\ran -\overline{\lan f\ran}]=L[\lan f\ran]$ and the nonlinear enhanced dissipation \eqref{nl_ED} to obtain
\begin{align*}
\frac{1}{2}\frac{d}{dt}\lf\|\lan f\ran -\overline{\lan f\ran}\rg\|_{L^2}^2\leq& -\frac{\nu}{2}\lf\|\pa_\theta\lf(\lan f\ran-\overline{\lan f\ran} \rg)\rg\|_{L^2}^2+\frac{C\kappa^2}{\nu}\|L[\lan f\ran -\overline{\lan f\ran}]\|_{L^2}^2\\
&+\frac{C\kappa^2}{\nu}\|\lan f\ran -\overline{\lan f\ran}\|_{L^2}^2\|L[\lan f\ran -\overline{\lan f\ran}]\|_{L^\infty}^2+\frac{C\kappa^2}{\nu}\|f_{0;\nq}\|_{L^2}^2e^{-\delta\nu^{1/2}t}\\
\leq &-\lf(\frac{\nu}{2}-\frac{C\kappa^2}{\nu}\rg)\|\lan f\ran -\overline{\lan f\ran}\|_{L^2}^2+\frac{C\kappa^2}{\nu}\|f_{0;\nq}\|_{L^2}^2e^{-\delta\nu^{1/2}t}.
\end{align*}
We can see that there exists a constant $C_*$ such that if
\begin{align}
\nu\geq C_*\kappa,
\end{align}
the quantity $\lf\|\lan f\ran -\overline{\lan f\ran}\rg\|_{L^2}$ decays to zero as $t\rightarrow\infty$. 
\end{proof}

\begin{proof}[Proof of Theorem \ref{thm_3}, {\bf Part c)}]
Now, by Proposition 3.1 of \cite{FrouvelleLiu12}, this limiting function $f_\infty\in C^4$ with $\mathcal{D}[f_\infty]=0$ solves the equation \eqref{D=0}. Then the argument in the Appendix \ref{sec:g_s} yields the result.\myb{Check!}

\end{proof}

\appendix
\section{Linear \sima{Enhanced} Dissipation and Inviscid Damping}\label{App_A}
In this section, we study the simplified equation $\eqref{eq:bsc_1}_{\kappa=0}$ and derive the enhanced dissipation estimate. By  implementing the Fourier transform in the $\bx$ variables, one obtains the following $\bk$-by-$\bk$ equation  
\begin{align}\label{eq:bsc_k}
\pa_t \wh \eta_\bk +v i\bp\cdot\bk \wh \eta_\bk=\nu \pa^2_\theta \wh \eta_\bk,\quad \bk=(k_1,k_2),\quad \wh \eta_\bk(t=0,\bp)=\wh \eta_{0;\bk}(\bp).
\end{align}
Throughout the paper, we will use $\dss |\bk|=\sqrt{k_1^2+k_2^2}$ to denote the length of the vector $\bk.$ 
The remainder $\eta_\nq$, as defined in \eqref{x-avrg_rem}, can be decomposed as follows
\begin{align}
\eta_\nq(t,\bx,\bp)= \sum_{\bk\neq (0,0)} \wh \eta_{\bk}(t,\bp) e^{i\bk \cdot\bx}. 
\end{align}
The main goal of this section is to prove the following theorem.
\begin{theorem}\label{thm:led}
Consider solutions $\wh \eta_\bk\in C^1([0,\infty);H_\theta^2)$ to the equation \eqref{eq:bsc_k}. There exists a universal threshold $0<\nu_0\leq 1$ such that if $0<\nu\leq \nu_0$, then the following enhanced dissipation estimate holds
\begin{align}\label{ED}
\lf\|\wh \eta_\bk(t)\rg\|_{L^2}\leq C_0\lf\|\wh \eta_{0;\bk}\rg\|_{L^2}\exp\{-\delta_0 \nu^{1/2}|\bk|^{1/2}t\},\quad \forall t\geq 0.
\end{align}
Here, $C_0>1, \, \delta_0\in(0,1)$ are universal constants. 
\end{theorem}
\begin{remark}
If the speed $v(t)$ is invariant in time, i.e., $v'\equiv 0$, then there exists an alternative proof of the theorem using the resolvent analysis, see, e.g., \cite{FengShiWang22}. Here we employ the machinery of the hypocoercivity and extend the result to the time-dependent setting. 
\end{remark}

The proof of the theorem consists of several lemmas. The main object of study is the Hypocoercivity functional \cite{villani2009,BCZ15}
\begin{align}\label{F}
\quad \quad F[\wh \eta_\bk]:=&\|\wh \eta_\bk\|_2^2+\al\zeta_\bk\nu^{1/2}|\bk|^{-1/2}\|\pa_\theta \wh \eta_\bk\|_2^2+\beta\zeta_\bk^2 |\bk|^{-1}\Re\lan i(- k_1\sin\theta+ k_2\cos\te)  \wh \eta_\bk , \pa_\theta \wh \eta_\bk \ran\\
&\quad +\gamma\zeta_\bk^3\nu^{-1/2}|\bk|^{-1/2}\|(- k_1\sin\theta+k_2\cos\te ) \wh \eta_\bk\|_2^2,\qquad \zeta_\bk:=\min\{1, \nu^{1/2}|\bk|^{1/2}t\}.
\end{align}Here $\lan f,g\ran=\dss \int_{{\mathbb T}} f\overline{g}d\theta$ and $\al,\beta,\gamma$ are constants to be chosen. The spatial weights $i(-\sin\theta k_1+\cos\theta k_2)$ can be interpreted as the commutator of the conservative part and the dissipation part of the equation \eqref{eq:bsc_k}, i.e., $-[i\bp\cdot \bk,\partial_\theta]$. A similar version of the  Hypocoercivity functional \eqref{F} is already introduced in the papers of Albritton-Ohm \cite{AlbrittonOhm22}, and Coti Zeliati-Dietert-Gerard Varet  \cite{CotiZelatiDietertGerardVaret22} \sima{in two and three dimensions}. \sima{However, the explicit formulas are more explicit in a two-dimensional setting, hence we decided to carry out the details here. }
Since the equation \eqref{eq:bsc_k} is linear, the dynamics of the $\wh \eta_\bk$'s will not interfere with each other. Hence, we will use the simplified notation $ \eta:=\wh \eta_{\bk}$ when there is no cause for confusion. To further simplify the notation, we introduce the quantity 
\begin{align}\label{theta_k}
(\cos(\theta_\bk),\sin(\theta_\bk)):=\lf(\frac{k_1}{|\bk|},\frac{k_2}{|\bk|} \rg),\quad \lf(-\sin(\theta)\frac{k_1}{|\bk|}+\cos(\theta)\frac{k_2}{|\bk|}\rg)=-\sin(\theta-\theta_\bk),\quad |\bk|=\sqrt{k_1^2+k_2^2}.
\end{align}
As a result, the functional $F[\wh \eta_\bk](t)$ can be rewritten as follows:
\begin{align}\label{F_fnct}
F[\wh \eta_\bk](t)=&\|\wh \eta_\bk\|_{L^2}^2+\al\zeta_\bk\nu^{1/2}|\bk|^{-1/2}\|\pa_\theta \wh \eta_\bk\|_{L^2}^2-\beta\zeta_\bk^2 \Re\lan i\sin(\theta-\theta_\bk) \wh \eta_\bk , \pa_\theta \wh \eta_\bk \ran\\
&\quad+\gamma\zeta_\bk^3\nu^{-1/2}|\bk|^{1/2}\| \sin(\theta-\theta_\bk)\wh \eta_\bk\|_{L^2}^2.
\end{align}
We will prove two main lemmas concerning the functional \eqref{F_fnct}.
\begin{lem}[Comparison]\label{lem:eqv}
Assume the relation
    \begin{align}\label{eqv_req}
	 \beta^2 \leq \al\gamma. 
    \end{align}
    Then, the following equivalence relation holds
    \begin{align}\label{eqv}
        \|&\eta_\bk\|_2^2 + \frac{1}{2}\lf(\al\zeta_\bk\lf(\frac{\nu}{|\bk|}\rg)^{1/2}\|\pa_\theta \wh \eta_\bk\|_2^2 + \gamma\zeta_\bk^3\lf(\frac{\nu}{|\bk|}\rg)^{-1/2}\|\sin(\theta-\theta_\bk)\eta_\bk\|_2^2\rg) \\
         \leq& F[\eta_\bk]\leq \|\eta_\bk\|_2^2 + \frac{3}{2}\lf(\al\zeta_\bk\lf(\frac{\nu}{|\bk|}\rg)^{1/2}\|\pa_\theta \eta_\bk\|_2^2 + \gamma\zeta_\bk^3\lf(\frac{\nu}{|\bk|}\rg)^{-1/2}\|\sin(\theta-\theta_\bk)\eta_\bk\|_2^2\rg).
\end{align}
\end{lem}
\begin{proof}
   We apply H\"older inequality and Young's inequality,
    \begin{align}
        F[\eta] \leq& \|\eta\|_2^2 + \al\zeta_\bk\nu^{1/2}|\bk|^{-1/2}\|\pa_\theta \eta\|_2^2 + \beta\zeta_\bk^2\|\sin(\theta-\theta_\bk)\eta\|_2^2\|\pa_\theta \eta\|_2^2 + \gamma\zeta_\bk^3\nu^{-1/2}|\bk|^{1/2}\|\sin(\theta-\theta_\bk)\eta\|_2^2\\
        \leq& \|\eta\|_2^2 + \frac{3\al}{2}\zeta_\bk\nu^{1/2}|\bk|^{-1/2}\|\pa_\theta \eta\|_2^2 + \lf(\gamma + \frac{\beta^2}{2\al}\rg)\zeta_\bk^3\nu^{-1/2}|\bk|^{1/2}\|\sin(\theta-\theta_\bk)\eta\|_2^2.
    \end{align}
    Similarly, we have the lower bound,
    \begin{align}
        F[\eta] \geq \|\eta\|_2^2 + \frac{\al}{2}\zeta_\bk\nu^{1/2}|\bk|^{-1/2}\|\pa_\theta \eta\|_2^2 + \lf(\gamma - \frac{\beta^2}{2\al}\rg)\zeta_\bk^3\nu^{-1/2}|\bk|^{1/2}\|\sin(\theta-\theta_\bk)\eta\|_2^2.
    \end{align}
    Since \eqref{eqv_req} implies that $\frac{\beta^2}{2\al} \leq \frac{\gamma}{2}$, we obtain \eqref{eqv}.
\end{proof}

The second main lemma concerning the functional $F$ reads as follows.
\begin{lem}[Enhanced Dissipation]\label{lem:LED}
There exists a choice of parameters $\al, \beta,\gamma$ such that the time derivative of the functional \eqref{F_fnct} is bounded as follows
\begin{align}\label{led}
\frac{d}{dt}F[\wh \eta_\bk](t)\leq \mathbbm{1}_{t\leq \nu^{-1/2}|\bk|^{-1/2}}C \nu^{1/2}|\bk|^{1/2}F[\wh \eta_\bk]-2\delta\nu^{1/2}|\bk|^{1/2}F[\eta_\bk],\quad \forall t\in[0,\infty). 
\end{align}
Here $C>1,\ \delta\in (0,1)$ are universal constants. 
\end{lem}

With Lemma \ref{lem:eqv} and Lemma \ref{lem:LED}, the proof of Theorem \ref{thm:led} is direct. 
\begin{proof}[Proof of Theorem \ref{thm:led}]
We integrate the relation \eqref{led} on the time interval $[0,\nu^{-1/2}|\bk|^{-1/2}]$ to obtain that 
\begin{align}
F[\eta_\bk(t)]\leq e^{C}F[\wh \eta_{0;\bk}],\quad \forall t\in[0,\nu^{-1/2}|\bk|^{-1/2}]. \label{short_t}
\end{align}
For $t>\nu^{-1/2}|\bk|^{-1/2}$, we integrate in time and invoke the bound \eqref{short_t} to obtain that
\begin{align}
F[\wh \eta_\bk(t)]\leq& F[\wh \eta_\bk(\nu^{-1/2}|\bk|^{-1/2})]\exp\{-2\delta\nu^{1/2}|\bk|^{1/2}(t-\nu^{-1/2}|\bk|^{-1/2})\}\\
\leq &e^{C+2\delta}F[\wh \eta_{0;\bk}]\exp\{-2\delta \nu^{1/2}|\bk|^{1/2}t\}=e^{C+2\delta}\|\wh \eta_{0;\bk}\|_{L^2}^2\exp\{-2\delta \nu^{1/2}|\bk|^{1/2}t\}.
\end{align}
Since $F[\eta_\bk]\geq \|\eta_\bk\|_{2}^2$, we have obtained the result.
\end{proof}
The most technical part of the proof then boils down to the justification of Lemma \ref{lem:LED}. 
\begin{proof}[Proof of Lemma \ref{lem:LED} ]
We take the time derivative of the $F[\eta(t)]$ functional,
\begin{align}\label{dtF}
\frac{d}{dt}F[\eta(t)]=& \frac{d}{dt}\|\eta\|_2^2+\al\nu^{1/2}|\bk|^{-1/2}\frac{d}{dt}(\zeta_\bk\|\pa_\theta \eta\|_2^2)-\beta \frac{d}{dt}(\zeta_\bk^2\Re\lan i\sin(\theta-\theta_\bk)\eta , \pa_\theta \eta\ran)\\
\n &+\gamma\nu^{-1/2}|\bk|^{1/2}\frac{d}{dt}(\zeta_\bk^3\|\sin(\theta-\theta_{\mathbf{k}})\eta\|_2^2)\\
=:&T_{L^2\n}+T_\al+T_\beta+T_\gamma.
\end{align}

Now we explicitly estimate each term in the expression \eqref{dtF}. The $T_{L^2}$-term is direct:
\begin{align}
\label{TL2}T_{L^2}=-2\nu \|\pa_\theta \eta\|_2^2.
\end{align}
To estimate the $T_\al$-term, we have 
\begin{align}\label{Tal}\quad \
T_\al
\leq&\al\nu\|\pa_\theta \eta\|_{2}^2-2\zeta_\bk\al\nu^{3/2}|\bk|^{-1/2}\|\pa_{\theta}^2\eta\|_2^2+\frac{\beta}{2}|\bk| v^2\zeta_\bk^2\|\sin(\theta-\theta _\bk)\eta\|_{2}^2+\frac{2\al^2}{\beta}\nu\|\pa_{\theta}\eta\|_2^2. 
\end{align}
To estimate the $T_\beta$-term, we write out the expression explicitly
\begin{align}\label{T_beta123} 
T_\beta=&-2\beta \zeta_\bk\zeta_\bk'\Re\lan i\sin(\theta-\theta_\bk)\eta,\pa_\theta \eta\ran-{\beta}\zeta_\bk^2 \Re \lan i\sin(\theta-\theta_\bk) \pa_t \eta, \pa_\theta \eta\ran\\
&- \beta \zeta_\bk^2 \Re\lan i\sin(\theta-\theta_\bk)  \eta, \pa_t\pa_\theta \eta\ran\\
=:& T_{\beta,1}+T_{\beta,2}+T_{\beta,3}.\n
\end{align}
We estimate the $T_{\beta,1}$-term as follows
\begin{align}
T_{\beta,1}\leq 2\beta \nu^{1/2}|\bk|^{1/2} \zeta_\bk\|\sin(\theta-\theta_\bk)\eta\|_{L^2}\|\pa_\theta \eta\|_{L^2}\leq \frac{1}{16}\beta |\bk|\zeta_\bk^2\|\sin(\theta-\theta_\bk)\eta\|_2^2+16\beta \nu\|\pa_{\theta}\eta\|_{2}^2.
\end{align}
The $T_{\beta,2}$-term can be estimated as follows 
\begin{align}
\n T_{\beta,2}
=&-{\beta}\zeta_\bk^2\Re \lan  i\sin(\theta-\theta_\bk)(\nu\pa_\theta^2\eta-vi\bp\cdot \bk \eta),  \pa_\theta \eta\ran\\
\leq &{\nu\beta}\zeta_\bk^2\|\pa_\theta^2\eta\|_{2}\|\sin(\theta-\theta_\bk)\pa_\theta \eta\|_2-{\beta}v\zeta_\bk^2\Re\lan \sin (\theta -\theta_\bk)\bp\cdot\bk \eta,\pa_\theta \eta\ran. 
\end{align}
Now we estimate the $T_{\beta,3}$-term in \eqref{T_beta123} using integration by parts,  H\"older inequality and Young's inequality, 
\begin{align}\n
T_{\beta,3}
& \leq \nu\beta\zeta_\bk^2\|\pa_\theta^2\eta\|_{2}\|\pa_\theta(\sin(\theta-\theta_\bk) \eta)\|_2+{\beta} v\zeta_\bk^2\Re\lan \sin (\theta-\theta_\bk) \eta,(\bp\cdot \bk \pa_\theta \eta+(-\sin\theta k_1 + \cos\theta k_2)\eta)\ran\\
&\leq\nu\beta\zeta_\bk^2\|\pa_\theta^2\eta\|_{2}(\|\sin(\theta-\theta_\bk)\pa_\theta \eta\|_2+\|\eta\|_2)+{\beta} v\zeta_\bk^2\Re\lan \sin (\theta-\theta_\bk) \eta,\bp\cdot \bk \pa_\theta \eta\ran\\
&\quad-\beta v\zeta_\bk^2|\bk|\|\sin(\theta-\theta_\bk)\eta\|_2^2. 
\end{align}
Now if we sum the three terms together, we obtained the following bound for $T_\beta$
\begin{align}\label{Tbeta}
\qquad T_\beta
\leq& \frac{1}{16}\beta |\bk|\zeta_\bk^2\|\sin(\theta-\theta_\bk)\eta\|_2^2+16\beta \nu\|\pa_{\theta}\eta\|_{2}^2
+\al\zeta_\bk\nu^{3/2}|\bk|^{-1/2}\|\pa_\theta^2\eta\|_{2}^2\\
&+\frac{\beta^2}{\al}\nu^{1/2}|\bk|^{1/2}\zeta_\bk^3\lf(\|\sin(\theta-\theta_\bk)\pa_\theta \eta\|_2^2+\frac{1}{4}\|\eta\|_2^2\rg)- {\beta}v\zeta_\bk^2{|\bk|}\lf\| \sin (\theta-\theta_\bk) \eta\rg \|_2^2. 
\end{align}
Now the $T_\gamma$-term can be estimated as follows:
\begin{align}\label{Tga}\quad
T_\gamma
\leq&3\gamma\zeta_\bk^2\mathbbm{1}_{t\leq \nu^{-1/2}|\bk|^{-1/2}}|\bk|\|\sin(\theta-\theta_\bk)\eta\|_2^2 -\gamma\nu^{1/2}|\bk|^{1/2}\zeta_\bk^3\|\sin(\theta-\theta_\bk)\pa_{\theta }\eta\|_2^2\\
&+\gamma\nu^{1/2}|\bk|^{1/2}\zeta_\bk^3\|\eta\|_2^2.
\end{align} 
Now summing up the $T_{L^2}$-estimate \eqref{TL2}, $ T_\al$-estimate \eqref{Tal}, $T_\beta$-estimate \eqref{Tbeta}, $T_\gamma$-estimate \eqref{Tga} as in the decomposition \eqref{dtF}, we end up with the estimate
\begin{align}
\frac{d}{dt}F[\eta]\leq&\lf(-2+\al+\frac{2\al^2}{\beta}+16\beta\rg)\nu \|\pa_\theta \eta\|_2^2-\al\nu^{3/2}|\bk|^{-1/2}\zeta_\bk\|\pa_\theta ^2\eta\|_2^2\\
&+\lf(-v+\frac{1}{2}v^2+\frac{1}{16}+\frac{3\gamma}{\beta}\mathbbm{1}_{t\leq\nu^{-1/2}|\bk|^{-1/2}}\rg)\beta |\bk|\zeta_\bk^2\|\sin(\theta-\theta_\bk)\eta\|_2^2\\
&+\lf(-1+\frac{\beta^2}{\al\gamma}\rg)\gamma\nu^{1/2}|\bk|^{1/2}\zeta_\bk^3\|\sin(\theta-\theta_\bk)\pa_\theta \eta\|_2^2+\lf(\frac{\beta}{4\al}+\frac{\gamma}{\beta}\rg)\beta\zeta_\bk^2\nu^{1/2}|\bk|^{1/2}\|\eta\|_2^2.
\end{align} 
We choose $\al$, $\beta$, $\gamma$ as follows 
    \begin{align}\label{chc_al_ga}
\al = \frac{\beta^{1/2}}{4}, \quad\quad \gamma = 4\beta^{3/2},\quad\quad \beta\leq \frac{1}{4096} .
    \end{align}
Now we have that 
\begin{align*}
\frac{d}{dt}F[\eta]\leq&-\nu \|\pa_\theta \eta\|_2^2-\frac{1}{8}\beta |\bk|\zeta_\bk^2\|\sin(\theta-\theta_\bk)\eta\|_2^2+5\beta^{3/2}\zeta_\bk^2\nu^{1/2}|\bk|^{1/2}\|\eta\|_2^2.
\end{align*}
To bound the right-hand side in terms of the $F[\eta]$, we invoke the following spectral inequality
\begin{align}
\nu^{1/2}|\bk|^{1/2}\|\eta_\bk\|_{L_\theta^2}^2\leq\nu\|\pa_\theta \eta_\bk\|_{L_\theta^2}^2+\mathfrak{C}_{\text{spec}}|\bk|\|\sin(\theta-\theta_\bk)\eta_\bk\|_{L_\theta^2}^2.
\end{align}
The proof of the inequality can be found in Proposition 2.7 of \cite{BCZ15}, and Lemma 3.1 of \cite{CobleHe23}. 
As a result, we have 
\begin{align*}
\frac{d}{dt}F[\eta]\leq&-\frac{\nu}{2} \|\pa_\theta \eta\|_2^2+\lf(-\frac{1}{16\mathfrak{C}_{\text{spec}}}+5\beta^{1/2}\rg)\beta \nu^{1/2}|\bk|^{1/2}\zeta_\bk^2\|\eta\|_2^2-\frac{1}{16}\beta |\bk|\zeta_\bk^2\|\sin(\theta-\theta_\bk)\eta\|_2^2.
\end{align*}
Hence, we can choose
    \begin{align}
    \beta=\beta(\mathfrak{C}_{\mathrm{spec}})\leq\frac{1}{4096},\label{chc_beta}
    \end{align}
    small enough such that 
\begin{align}\label{energy_dissipation}
\frac{d}{dt}F[\eta]
\leq& -\delta( \mathfrak C_\text{spec})\nu^{1/2}|\bk|^{1/2}\mathbbm{1}_{t\geq \nu^{-1/2}|\bk|^{-1/2}}F[\eta].
\end{align}
Here in the last line, we have invoked the relation \eqref{eqv}. This concludes the proof.
\end{proof}

\begin{lem}[Linear Inviscid Damping] Consider solutions to the equation $\eqref{eq:bsc_k}_{|\bk|\neq0}$, then the following linear inviscid damping estimate holds: 
\begin{align}
    \|\eta_\nq(t)\|_{H^{-1}}\leq C\min\lf\{\frac{\nu^{1/4}}{\min\{1,\nu^{1/4}t^{1/2}\}},e^{-\delta_0\nu^{1/2}t}\rg\}\|\eta_\nq(0)\|_{H^1},\quad\forall t\geq0. \label{linear_ID}
\end{align}
Here, $C\geq 1$ is a universal constant, and the $\delta_0$ is defined in \eqref{ED}.  
\end{lem}
\begin{proof}
    
We decompose the proof into three steps. The basic ideas of the proof are from a series of works \cite{CotiZelatiDietertGerardVaret22,CotiZelatiDietertVaret24}. Since our setting differs from theirs and the quantitative behavior is also different, we provide the details here.

\noindent
{\bf Step \# 1: Preliminaries: }First of all, we observe that it is enough to consider the $t\geq 1$ case because within the time interval $t\in[0,1]$, the estimate \eqref{linear_ID} is a direct consequence of the relation $\|\eta_\nq\|_{H^{-1}}\leq C\|\eta_\nq\|_{L^2}$ and the nonexpansive nature of the quantity $\|\eta_\nq\|_{L^2}$. The key mathematical object that leads to the inviscid damping estimate is the following vector fields (adapted to each $\bx$-Fourier mode) introduced on page 11 of  \cite{CotiZelatiDietertVaret24}:
\begin{align}
    \begin{cases}\begin{cases}&J_\bk^+ \eta_\bk= A_\bk^+(t)\pa_\theta \eta_\bk+i\lf(\nu|\bk|^{-1}\rg)^{-1/2}B_\bk^+(t)\pa_\theta\bp(\theta) \cdot \frac{\bk}{|\bk|} \eta_\bk\\
    &\hspace{1.05 cm}=A_\bk^+(t)\pa_\theta \eta_\bk-i\lf(\nu|\bk|^{-1}\rg)^{-1/2}B_\bk^+(t)\sin(\theta-\theta_\bk) \eta_\bk,\\
    &  A_\bk^+(t)=\frac{1}{2}\lf(1+e^{-2(1-i)\sqrt{\nu|\bk|}t}\rg),\quad B_\bk^+(t)=\frac{1+i}{4}\lf(1-e^{-2(1-i)\sqrt{\nu|\bk|}t}\rg);\end{cases}\\
    \begin{cases}
    &J_\bk^- \eta_\bk= A_\bk^-(t)\pa_\theta \eta_\bk+i\lf(\nu|\bk|^{-1}\rg)^{-1/2}B_\bk^-(t)\pa_\theta\bp(\theta) \cdot \frac{-\bk}{|\bk|} \eta_\bk\\
    &\hspace{1.05cm}=A_\bk^-(t)\pa_\theta \eta_\bk-i\lf(\nu|\bk|^{-1}\rg)^{-1/2}B_\bk^-(t)\sin(\theta- \theta_\bk^-) \eta_\bk,\quad  \theta_\bk^- :=\theta_\bk+\pi(\text{mod }2\pi),\\
     &A_\bk^-(t)=\frac{1}{2}\lf(1+e^{-2(1+i)\sqrt{\nu|\bk|}t}\rg),\quad B_\bk^-(t)=\frac{i-1}{4}\lf(1-e^{-2(1+i)\sqrt{\nu|\bk|}t}\rg).
     \end{cases}\end{cases}
\end{align}
Here, the augmented angle $\theta_\bk$ is  defined in \eqref{theta_k}. %
The motivation for designing these vector fields is to approximate the vector field $\pa_\theta+it\pa_\theta(\bp\cdot \bk)$ that commutes with the inviscid dynamics ($\nu=0$). However, thanks to the diffusion term $\nu\pa_\te^2$, the commutator between the vector fields $J^\pm_\bk$ and the equation is nontrivial. However, by carefully tuning the coefficients, it can be guaranteed that these commutators vanish at one of the critical points of the function $\bp(\cdot)\cdot\bk\in C^\infty(\Torus)$. Then, one can use structures of the enhanced dissipation functional to control these commutators in a suitable domain. It turns out that this is sufficient to derive mixing. We observe that \begin{align}|A_\bk^\pm(t)|\approx 1,\quad |B_\bk^\pm(t)|\approx \min\{\nu^{1/2}|\bk|^{1/2}t,1\}=\zeta_\bk(t). \label{ab_equiv}
\end{align}To simplify the notation, we define $\wh {\bk}=\bk/|\bk|$ as in \cite{CotiZelatiDietertVaret24}. For the vector field $J_\bk$, there is an associated cutoff function $\chi_\bk$. It is a smooth cutoff function that is $1$ near $\theta_\bk$ and $0$ near $\theta_\bk+\pi(\text{mod\ } 2\pi)$. Through direct computation, we obtain the following equation for the $J_\bk \eta_\bk$:
\begin{align}
   \label{J_eta_eq}\hspace{1cm}\pa_tJ_\bk^+ \eta_\bk+i\bp\cdot \bk J_\bk^+\eta_\bk&=\nu\pa_{\te}^2 J_\bk^+\eta_\bk-|\bk|(1-i)(\nu|\bk|^{-1})^{1/2}J_\bk^+ \eta_\bk \\
    &+2iB_\bk^+\nu^{1/2}|\bk|^{1/2}\pa_\te((\bp\cdot \wh\bk-1)\eta_\bk)+iB_\bk^+\nu^{1/2}|\bk|^{1/2}\sin(\te-\te_\bk)\eta_\bk;\\
\label{wt_J_eta_eq}    \pa_tJ_\bk^- \eta_\bk+i\bp\cdot \bk  J_\bk^-\eta_\bk&=\nu\pa_{\te}^2 J_\bk^-\eta_\bk-|\bk|(1+i)(\nu|\bk|^{-1})^{1/2}J_\bk^- \eta_\bk \\
    &-2iB_\bk^-\nu^{1/2}|\bk|^{1/2}\pa_\te((\bp\cdot \wh\bk+1)\eta_\bk)+iB_\bk^-\nu^{1/2}|\bk|^{1/2}\sin(\te-\te_\bk^-)\eta_\bk.
\end{align}
We note that the last two terms on the right-hand side of the equation \eqref{J_eta_eq} and \eqref{wt_J_eta_eq} vanish at point $\theta=\theta_\bk$ and $\theta=\theta_\bk^-$, respectively. Since the two equations have similar structures, we only focus on \eqref{J_eta_eq}. Moreover, we drop the superscript $(\cdots)^{+}$ in $J_\bk^+,\ A_\bk^+,\ B_\bk^+$ to simplify the notations. This concludes the step.

\noindent
{\bf Step \# 2: Estimate of the vector fields: }
Now we define two smooth cutoff functions $\chi_\bk, \ \wt{\chi}_\bk\in C^\infty(\Torus)$ supported around $\wh \bk$ such that both of them are zero near $-\bk/|\bk|$ and $|\chi_\bk|+|\pa_\te \chi_\bk|\leq C\wt{\chi}_\bk\leq C$, as in the proof of Lemma 3.4 in \cite{CotiZelatiDietertVaret24}. Moreover, the norms of these cutoffs are independent of $\bk. $ One can do the same for the vector field $ J_\bk^-.$ Now we implement the $L^2$-energy estimate of the truncated quantity $\chi_\bk J_\bk \eta_\bk$:
\begin{align}\label{D2T3}
    \frac{1}{2}\frac{d}{dt}\|\chi_\bk J_\bk \eta_\bk\|_{L^2}^2\leq&-\nu\|\chi_\bk \pa_\te J_\bk\eta_\bk\|_{L^2}^2-\nu^{1/2}|\bk|^{1/2}\|\chi_{\bk}J_\bk\eta_\bk\|_{L^2}^2\\
    &+2\nu\|\chi_\bk\pa_\te J_\bk\eta_\bk\|_{L^2}\|\pa_\te\chi_\bk J_\bk\eta_\bk\|_{L^2}\\
    &+2|B_\bk|\nu^{1/2}|\bk|^{1/2}\int \lf|\pa_\te(J_\bk \eta_\bk\chi_\bk^2)\rg|\lf|(\bp\cdot\wh \bk-1)\eta_\bk\rg|d\te\\
    &+|B_\bk|\nu^{1/2}|\bk|^{1/2}\int \lf|J_\bk \eta_\bk\chi_\bk^2\sin(\te-\te_\bk)\eta_\bk\rg|d\te\\
   =: &-D_1-D_2+T_1+T_2+T_3.
\end{align}
Thanks to the relation \eqref{energy_dissipation}, we have that 
\begin{align}\label{ED_result}
&\int_0^{\infty} \nu \|\pa_\theta \eta_\bk\|_2^2+\nu^{1/2}|\bk|^{1/2}\zeta_\bk^2\|\eta_\bk\|_2^2+ |\bk| \zeta_\bk^2\|\sin(\theta-\theta_\bk)\eta_\bk\|_2^2 dt \\
&\qquad\leq  C(\mathfrak C_\text{spec}) F[\eta_{\bk}(0)]=C(\mathfrak C_\text{spec})\|\eta_{\bk}(0)\|_{L^2}^2.
\end{align}
Next, we use this time integrability condition to estimate the right-hand side of \eqref{D2T3}. For the $T_1$-term, we estimate it with the relation $J_\bk=A_\bk\pa_\te-i B_\bk(\nu|\bk|^{-1})^{-1/2}\sin(\te-\te_\bk)$ and the equivalence \eqref{ab_equiv} as follows:
\begin{align}\label{T_1_est}
     &\|T_1\|_{L_t^1}\leq\frac{1}{4}\|D_1\| _{L_t^1}+C\nu\int\lf( |A_\bk|^2\|\pa_\te \eta_\bk\|_{L_\te^2}^2+ |B_\bk|^2(\nu|\bk|^{-1})^{-1}\|\sin(\te-\te_\bk)\eta_\bk\|_{L_\te^2}^2\rg)dt\\
    &\leq \frac{1}{4}\|D_1\| _{L_t^1}+C \int\lf(  \nu\|\pa_\te \eta_\bk\|_{L_\te^2}^2+ \zeta_\bk^2|\bk|\|\sin(\te-\te_\bk)\eta_\bk\|_{L_\te^2}^2\rg)dt\leq \frac{1}{4}\|D_1\| _{L_t^1}+C\|\eta_\bk(0)\|_{L_\te^2}^2.
\end{align}
In the last line, we use the estimate \eqref{ED_result}. 

Next we estimate the $T_2$-term in \eqref{D2T3} with the fact that on the support of $\chi_\bk,$ $|\bp(\te)\cdot \wh \bk-1|\leq C|\pa_\te(\bp(\te)\cdot \wh \bk)|=C|\sin(\te-\te_\bk)|$:
\begin{align}\label{T_2_est}
    &\|T_2\|_{L_t^1}\leq 2\int \zeta_\bk\nu^{1/2}|\bk|^{1/2}\int  \chi_\bk\Big(|\pa_\te J_\bk \eta_\bk|  \chi_\bk+|\pa_\te\chi_\bk J_\bk \eta_\bk|\Big)|\sin(\te-\te_\bk)||\eta_\bk|d\te dt\\
    &\leq 2\int\zeta_\bk\nu^{1/2}|\bk|^{1/2}\|\chi_\bk\pa_\te J_\bk\eta_\bk\|_{L_\te^2}\|\sin(\te-\te_\bk)\eta_\bk\|_{L_\te^2}dt\\
    &\quad+C\nu^{1/2}\int  \zeta_\bk\Big(|A_\bk|\|\pa_\te  \eta_\bk\|_{L_\te^2}+|B_\bk|\nu^{-1/2}|\bk|^{1/2}\|\sin(\te-\te_\bk)\eta_\bk\|_{L_\te^2}\Big)\||\bk|^{1/2}\sin(\te-\te_\bk)\eta_\bk\|_{L_\te^2} dt\\
    &\leq \frac{1}{4}\|D_1\|_{L_t^1}+C  \int\lf(  \nu\|\pa_\te \eta_\bk\|_{L_\te^2}^2+ \zeta_\bk^2|\bk|\|\sin(\te-\te_\bk)\eta_\bk\|_{L^2}^2\rg)dt\leq \frac{1}{4}\|D_1\| _{L_t^1}+C\|\eta_\bk(0)\|_{L^2}^2.
    \end{align}

Finally, the $T_3$-term in \eqref{D2T3} can be estimated in the same fashion as $T_2$:
\begin{align}\label{T_3_est}
    &\|T_3\|_{L_t^1}
    \leq C  \int\lf(  \nu\|\pa_\te \eta_\bk\|_{L_\te^2}^2+ \zeta_\bk^2|\bk|\|\sin(\te-\te_\bk)\eta_\bk\|_{L^2}^2\rg)dt
    \leq C\|\eta_\bk(0)\|_{L^2}^2.
\end{align}
Now we integrate the expression \eqref{D2T3} in time to get 
\begin{align}\|J_\bk \eta_\bk(t) \chi_\bk\|_{L^2}\leq C\|\eta_\bk(0)\|_{H^1}.\label{vec_ctrl}
\end{align}

\noindent
{\bf Step \# 3: Proof of mixing. } We follow the idea of the proof of Proposition 1.7 in the paper \cite{CotiZelatiDietertGerardVaret22}. For a general test function $F\in H^1(\Torus^3),$ we can rewrite the expression $\iint \eta_\nq F d\bx d\te$ as 
\begin{align}
\sum_{\bk\neq(0,0)}\int \eta_\bk (t,\te)\overline{F_\bk(\te)} d\te=&\sum_{\bk\neq(0,0)}\lf(\int \eta_\bk (t,\te)\overline{F_\bk(\te)}\chi_\bk  d\te+\int \eta_\bk (t,\te)\overline{F_\bk(\te)}(1-\chi_\bk)  d\te\rg)\\
=:&\sum_{\bk\neq (0,0)}\lf(I_{\bk;+}+I_{\bk;-}\rg).
\end{align}
 As explained in the paper \cite{CotiZelatiDietertGerardVaret22}, a symmetry consideration yields that it is enough to consider the first part of the expression. For the second part, one can use the vector field $ J_\bk^-$ and associated cutoffs to derive similar estimates. One can introduce another cutoff function $\chi_{\bk;\ep}$ such that it is $1$ in an $\ep$-neighborhood of $\theta_\bk$. Moreover, $\|\pa_\te \chi_{\bk;\ep}\|_{L^\infty}\leq C\ep^{-1}$. With this cutoff, we can further decompose the $I_{\bk;+}$ as follows
\begin{align}
   I_{\bk;+}=& \int_\Torus \eta_\bk \overline{F_\bk} \chi_{\bk;\ep}\chi_{\bk} d\te+\int_\Torus \eta_\bk \overline{F_\bk} (1-\chi_{\bk;\ep})\chi_{\bk} d\te
   =:I_{\bk;+}^{(1)}+I_{\bk;+}^{(2)}.
\end{align}
For the $I_{\bk;+}^{(1)}$-term, we estimate it using the length of the interval, the observation that $\|\eta_\bk(t)\|_{L^\infty}\leq \|\eta_\bk(0)\|_{L^\infty}$ and the Sobolev embedding:
\begin{align*}
    |I_{\bk;+}^{(1)}|\leq C\ep \|F_\bk\|_{L^\infty}\|\eta_\bk(0)\|_{L^\infty}\leq C\ep\|F_\bk\|_{H^1}\|\eta_\bk(0)\|_{H^1}.
\end{align*}
Next we estimate the $I_{\bk;+}^{(2)}$ term using the observation that $|\sin(\theta-\theta_\bk)|>0$ on this interval and $(\nu|\bk|^{-1})^{1/2}(J_\bk \eta_\bk-A_\bk\pa_\te \eta_\bk)=-iB_\bk\sin(\theta-\theta_\bk)\eta_\bk$:
\begin{align}
    &|I_{\bk;+}^{(2)}|=\lf|\int_\Torus \sin^2(\te-\te_\bk)\eta_\bk\frac{\overline{F_\bk}}{\sin^2(\te-\te_\bk)} (1-\chi_{\bk;\ep})\chi_{\bk} d\te\rg|\\
   & \leq\nu^{1/2}|\bk|^{-1/2}\lf|\int_\Torus \sin(\te-\te_\bk)\frac{J_\bk\eta_\bk-A_\bk\pa_\te\eta_\bk}{-i B_\bk}\frac{\overline{F_\bk}}{\sin^2(\te-\te_\bk)}(1-\chi_{\bk;\ep})\chi_{\bk} d\te\rg| \\
   &\leq \int_\Torus \nu^{1/2}|\bk|^{-1/2} |\sin(\te-\te_\bk)|\lf|\frac{J_\bk\eta_\bk}{B_\bk}\rg|\lf|\frac{\overline{F_\bk}}{ \sin^2(\te-\te_\bk)} \rg| (1-\chi_{\bk;\ep})\chi_{\bk} d\te\\
    &\quad+\nu^{1/2}|\bk|^{-1/2}\lf|\frac{A_\bk}{B_\bk}\rg| \lf|\int_\Torus\sin(\te-\te_\bk)\pa_\te\eta_\bk \frac{\overline{F_\bk}}{\sin^2(\te-\te_\bk)} (1-\chi_{\bk;\ep})\chi_{\bk} d\te\rg|=:T_4+T_5.
\end{align}
For the $T_4$-term, we estimate it with the Sobolev embedding and \eqref{ab_equiv} 
as follows:\begin{align}
    T_4\leq& \frac{C\nu^{1/2}}{|\bk|^{1/2}|B_\bk|}\|F_\bk\|_{L^\infty} \int_{|\te-\te_\bk|\geq \ep} \frac{|J_\bk \eta_\bk| }{|\sin(\te-\te_\bk)|} d\te 
    \leq  \frac{C\nu^{1/2}}{|\bk|^{1/2}\min\{\nu^{1/2}|\bk|^{1/2}t,1\}}\|F_\bk\|_{H^1}\|J_\bk\eta_\bk\|_{L^2}\ep^{-1/2}\\
    \leq &\frac{C\nu^{1/2}}{ |\bk|^{1/2}\min\{\nu^{1/2} t,1\}}\|F_\bk\|_{H^1}\|J_\bk\eta_\bk\|_{L^2}\ep^{-1/2},\quad \forall t\geq 1.
\end{align}
To estimate the second term $T_5$, we recall the relation \eqref{ab_equiv}, implement integration by parts, and estimate each resulting term as follows,
\begin{align}
&T_5\leq \lf|\frac{A_\bk}{B_\bk}\rg| \frac{\nu^{1/2}}{|\bk|^{1/2}}\lf|\int_\Torus\eta_\bk\pa_\te\lf( \frac{\overline{F_\bk}}{\sin (\te-\te_\bk)} (1-\chi_{\bk;\ep})\chi_{\bk}\rg) d\te\rg|\\
&\leq  \frac{C\nu^{1/2}}{|\bk|^{1/2}\min\{\nu^{1/2}t,1\} }\|\eta_\bk\|_{L^\infty}\|F_\bk\|_{H^1}\lf(\int \frac{(1-\chi_{\bk;\ep})\chi_{\bk}}{|\sin^2(\te-\te_\bk)|} d\te+\int \frac{ |\pa_\te\chi_{\bk;\ep}|\chi_\bk +|\pa_\te\chi_\bk|(1-\chi_{\bk;\ep})}{\lf|\sin(\te-\te_\bk)\rg|}d\te\rg)\\
&\leq  \frac{C\nu^{1/2}}{|\bk|^{1/2}\min\{\nu^{1/2} t,1\}}\|\eta_\bk(0)\|_{H^1}\|F_\bk\|_{H^1}\frac{1}{\ep}.
\end{align}
Hence, we observe that if we set $\ep=\left(\frac{\nu^{1/2}}{\min\{\nu^{1/2} t,1\}}\right)^{1/2}$ and invoke the bound \eqref{vec_ctrl}, the following estimate holds
\begin{align}
    |I_{\bk;+}|\leq\frac{C\nu^{1/4}}{ \min\{\nu^{1/4} t^{1/2},1\}}\|\eta_\bk(0)\|_{H^1}\|F_\bk\|_{H^1},\quad \forall t\geq 1.
\end{align}
Summing all the $\bk$-components, we obtain that 
\begin{align}\label{mix}
\lf|\sum_{\bk\neq(0,0)}\int \eta_\bk (t,\te)\overline{F_\bk(\te)} d\te\rg|\leq& \sum_{\bk\neq (0,0)}\frac{C \nu^{1/4}}{ \min\{\nu^{1/4}t^{1/2},1\}}\|\eta_\bk(0)\|_{H^1}\|F_\bk\|_{H^1}\\
&\leq \frac{C\nu^{1/4}}{\min\{\nu^{1/4}t^{1/2},1\}}\|\eta_\nq(0)\|_{H^1}\|F_\nq\|_{H^1}.
\end{align}
Moreover, thanks to the enhanced dissipation estimate \eqref{ED}, we have the following estimate for all time
\begin{align}
\lf|\sum_{|\bk|\neq 0}\int\eta_\bk(t,\te)\overline{F_\bk(\te)}d\te\rg|\leq C\|\eta_\nq(0)\|_{L^2}\exp\{-\delta_0\nu^{1/2}t\}\|F_\nq\|_{L^2}.\label{ed}
\end{align}
Combining the estimates \eqref{mix} and \eqref{ed}, we obtain the result.

\end{proof}
\section{Connection between Different Formulations}
In the paper \cite{FrouvelleLiu12}, the authors consider the following equation
\begin{align*}
\pa_t g=-\na_\bp\cdot((I_d-\bp\otimes \bp)J[g]g)+\tau\de_\bp g ,\quad \bp\in \mathbb{S}.
\end{align*}
Here $\na_\bp\cdot$ is the tangential divergence of the vector field and $\de_\bp$ is the Laplace-Beltrami operator. The operator $J[g]$ is defined  as follows:
\begin{align}
J[g]=\int_{\mathbb{S}}\bp g(\cdot, \bp)d\bp.
\end{align}
We would like to show that this is actually equivalent to our equation \eqref{effctv_dym}. First of all, we observe that the Laplace-Beltrami operator is equivalent to $\pa_{\theta}^2 g$. Next, we have that 
\begin{align}
J[g]=\int_{-\pi}^\pi \binom{\cos\theta'}{\sin\theta'} g(\cdot, \theta')d\theta'.
\end{align} 
Next we recall that on the unit circle $\na_\bp\cdot (F_1,F_2)=-\sin\theta\pa_\theta F_1+\cos\theta \pa_\theta F_2.$ Hence,
\begin{align}
-\na_\bp&\cdot ((I_d-\bp\otimes\bp) J[g]g)\\
=&-\na_\bp\cdot\lf(
\begin{pmatrix}
   \sin^2\theta& -\sin\theta\cos\theta \\
  -\sin\theta\cos \theta & \cos^2\theta \\
\end{pmatrix}\int_{-\pi}^\pi \binom{\cos\theta'}{\sin\theta'} g(t, \theta')d\theta'  g(t,\theta)\rg)\\
=&-\na_\bp\cdot\lf(
\int_{-\pi}^\pi \begin{pmatrix}
   \sin^2\theta\cos\theta' -\sin\theta\cos\theta \sin\theta' \\
  -\sin\theta\cos \theta\cos\theta'  +\cos^2\theta\sin\theta' \\
\end{pmatrix}g(t, \theta')d\theta'  g(t,\theta)\rg)\\
=&-\binom{-\sin\theta\pa_\theta}{\cos\theta\pa_\theta}\cdot\lf(
\int_{-\pi}^\pi \begin{pmatrix}
   -\sin\theta  \\
   \cos \theta \\
\end{pmatrix}\sin(\theta'-\theta)g(t, \theta')d\theta'  g(t, \theta)\rg)\\
=&\pa_\theta\lf(\int_{-\pi}^\pi \sin(\theta-\theta')g(t, \theta')d\theta'  g(t, \theta)\rg).
\end{align}
Combining all the computations above, we have that the equation analyzed in \cite{FrouvelleLiu12} is identical to the \eqref{effctv_dym} modulo changes in parameters. 

\section{The Sobolev Estimates of the Solution}
In this section, we use the multi-index notation 
\begin{align}
\pa_\bx^{i}\pa_\te^j=\pa_{x_1}^{i_1}\pa_{x_2}^{i_2}\pa_\te^j,\quad i=(i_1,i_2),\quad |i|=i_1+i_2,\quad |i,j|=i_1+i_2+j,\quad\binom{i}{i'}=\binom{i_1}{i_1'}\binom{i_2}{i_2'}.
\end{align}
Moreover, we denote $i'\leq i$ if $i_1'\leq i_1$ and $i_2'\leq i_2$.
We derive the following lemma.
\begin{lem}
Consider the solution $f$ to the equation \eqref{eq:bsc} initiated from data $f_0\in H^M_{\bx,\te}$. Assume that the $L^2$-norm of the solution is bounded, i.e.,
\begin{align}\label{L2_bnd}
\|f\|_{L^2_{\bx,\theta}}^2\leq C(1+\|f_0\|_{L_{\bx,\te}^2}^2)\nu^{-1/6}.
\end{align}
Then the following estimate  holds
\begin{align}
\label{Hn_bnd}
\|\pa_\bx^i \pa_\te^j f\|_{L^2}^2\leq& C \frac{\max\{1,\|f_0\|_{H_{\bx,\theta}^n}^2\}}{\nu^{4j/3+1/6}} \exp\lf\{C\nu^{2/3} t\rg\},\quad |i,j|=n\in\{1,2,\cdots, M\}.
\end{align}
Here the constant $C$ depends on the norm $\|\Phi\|_{W^{M,\infty}},\, \|\Psi\|_{W^{M,\infty}},\, i,\, j.$
\end{lem}
\begin{proof} We decompose the proof into several steps. 

\noindent
{\bf Step \# 1: Setup. }
We apply the induction argument to derive the bound \eqref{Hn_bnd}. The $n=0$ case is a natural consequence of the assumption \eqref{L2_bnd} and the constraint $\kappa\leq C\nu^{5/6}$. 
Assuming that \eqref{Hn_bnd} holds on the  $(n-1$)-th level, we would like to show that the estimate \eqref{Hn_bnd} holds. Thanks to the distinction between derivatives in $\bx$ and $\te$, we do another subinduction step. We will start with the case $|i|=n$, and sequentially increase the order of derivatives in $\theta$  and decrease the order of derivatives in $\bx$.

\noindent
{\bf Step \# 2: Induction. } 
We implement the $\dot H^{n}$-energy estimates with $|i|=n$,
\begin{align}
\frac{1}{2}&\frac{d}{dt}\sum_{|i|=n}\|\pa_\bx^i f \|_{L^2}^2
=-\nu\sum_{|i|=n}\|\pa_\theta \pa_\bx^{i} f\|_{L^2}^2+\kappa\sum_{|i|=n}\int\pa_\theta \pa_\bx^i f \  \pa_\bx^i (f L [f])\\
=& -\nu\sum_{|i|=n}\|\pa_\theta \pa_\bx^i f\|_{L^2}^2+\kappa\sum_{|i|=n}\int\pa_\theta \frac{|\pa_\bx^i f |^2}{2} \ L [f]+\kappa\sum_{i'\leq i}\binom{i}{i'}\int   \pa_\theta \pa_\bx^i f \ \pa_\bx^{i'}f\ \pa_\bx^{i-i'} L [f]   \\
\leq &-\frac{\nu}{2}\sum_{|i|=n}\|\pa_\theta \pa_\bx^i f\|_{L^2}^2+\frac{C\kappa^2}{\nu}\sum_{|i|=n}\|\pa_\te L[f]\|_{L^\infty}\|\pa_\bx^{i} f\|_{L^2}^2+\frac{C\kappa^2}{\nu}\|L [f]\|_{W^{n,\infty}}^2\sum_{i'\leq i,\ i'\nq i}\|\pa_\bx^{i'} f\|_{L^2}^2.
\end{align}
Now we apply the estimate that $\|L[ f]\|_{W^{n,\infty}}=\|(\Phi\Psi)*f\|_{W^{n,\infty}}\leq \|\Phi\Psi\|_{W^{n,\infty}}\|f\|_{L^1}$, $\kappa\leq \nu^{5/6}$ and the induction hypothesis to obtain that 
\begin{align*}
\frac{d}{dt}\sum_{|i|=n}\|\pa_\bx^i f(t)\|_{L^2}^2\leq&  C\nu^{2/3}\sum_{|i|=n}\|\pa_\bx^i f(t)\|_{L^2}^2+C\nu^{2/3}\frac{ \max\{1,\|f_0\|_{H_{\bx,\te}^{n-1}}^2\}}{\nu^{1/6}}\exp\lf\{C\nu^{2/3}t\rg\}.
\end{align*}Integrating in time yields that
\begin{align*}
\sum_{|i|=n}\|\pa_\bx^i f(t)\|_{L^2}^2\leq&  C\max\{1,\|f_0\|_{H^n}^2\}\exp\lf\{C\nu^{2/3}  t\rg\}+C\frac{\max\{1,\|f_0\|_{H^{n-1}}^2\}}{\nu^{1/6}}\nu^{2/3}t\exp\lf\{C{\nu}^{2/3} t\rg\}\\
\leq&C\max\{1,\|f_0\|_{H^n}^2\}\nu^{-1/6}\exp\lf\{2C{\nu}^{2/3} t\rg\}.
\end{align*}
This implies $\eqref{Hn_bnd}_{j=0}$ (with a larger constant $C$). 

Next, we implement the induction in $j\in\{1,2,\cdots, n\}$. Assume that the estimate \eqref{Hn_bnd} holds for the $(j-1)$-th level, we apply the energy estimate to derive that 
\begin{align*}
\frac{1}{2}&\sum_{|i|=n-j}\frac{d}{dt}\|\pa_\bx^i\pa_\te^j f\|_{L^2}^2\\
=&-\nu\sum_{|i|=n-j}\|\pa_\theta \pa_\bx^{i}\pa_\te^j f\|_{L^2}^2-\sum_{|i|=n-j}\int\pa_\bx^i\pa_\te^j f \ \pa_\te^j( \bp \cdot \na_\bx \pa_\bx^i f)  +\kappa\sum_{|i|=n-j}\int\pa_\theta \pa_\bx^i\pa_\te^j f \  \pa_\bx^i\pa_\te^j (f L [f])\\
=& -\nu\sum_{|i|=n-j}\| \pa_\bx^i\pa_\theta^{j+1} f\|_{L^2}^2-\sum_{|i|=n-j}\sum_{j'\leq j,\ j'\nq j}\binom{j}{j'}\int \pa_\bx^i\pa_\te^j f\ (\pa_\te^{j-j'}\bp)\cdot \na_\bx(\pa_\bx^{i}\pa_{\te}^{j'} f)\\
&+\kappa\sum_{|i|=n-j}\int\pa_\theta \frac{|\pa_\bx^i\pa_\te^j f |^2}{2} \ L [f]+\kappa\sum_{\substack{i'\leq i,\ j'\leq j\\ (i',j')\nq (i,j)}}\binom{i}{i'}\binom{j}{j'}\int \pa_\bx^i\pa_\theta ^{j+1} f \ \pa_\bx^{i'}\pa_\te^{j'} f\ \pa_\bx^{i-i'}\pa_\te^{j-j'} L [f]   \\
\leq &-\frac{1}{2}\nu\sum_{|i|=n-j}\|\pa_\bx^i\pa_\te^{i+1}f\|_{L^2}^2+C\lf(\sum_{|i|=n-j}\|\pa_\bx^i\pa_\te^j f\|_{L^2}^2\rg)^{1/2}\lf(\sum_{|i'|=n-j+1}\sum_{j'\leq j-1}\|\pa_\bx^{i'}\pa_\te^{j'} f\|_{L^2}^2\rg)^{1/2}\\
&+C\kappa \sum_{|i|=n-j}\|\pa_\bx^i\pa_\te^j f\|_{L^2}^2+C\kappa\sum_{\substack{i'\leq i,\ j'\leq j\\ (i',j')\nq (i,j)}} \|\pa_\bx^{i'}\pa_\te^{j'} f\|_{L^2}^2. 
\end{align*}
Thanks to the induction hypotheses, we have that there exists constant $C_0$ such that
\begin{align*}
\frac{d}{dt}\sum_{|i|=n-j}\|\pa_\bx^i\pa_\te^j f\|_{L^2}^2\leq &C_0\lf(\sum_{|i|=n-j}\|\pa_\bx^i\pa_\te^j f\|_{L^2}^2\rg)^{1/2}\frac{\max\{1,\|f_0\|_{H_{\bx,\te}^n}\}}{\nu^{ 2(j-1)/3+1/12}}\exp\lf\{C_0\nu^{2/3}t\rg\}\\
&+\nu^{5/6}\sum_{|i|=n-j}\|\pa_\bx^i\pa_\te^j f\|_{L^2}^2+\frac{C_0\kappa\max\{1,\|f_0\|_{H_{\bx,\te}^{n-1}}^2\}}{\nu^{4j/3+1/6}}\exp\lf\{C_0\nu^{2/3}t\rg\}.
\end{align*}
Now we consider the quantity:
\begin{align*}
G(t)^2:=\sum_{|i|=n-j}\|\pa_\bx^i\pa_\te^j f\|_{L^2}^2+\frac{\max\{1,\|f_0\|_{H_{\bx,\te}^n}^2\}}{\nu^{4j/3+1/6}}\exp\lf\{{C_0}\nu^{2/3}t\rg\}.
\end{align*} 
The differential inequality above yields that 
\begin{align*}
\frac{d}{dt}G\leq C_0\nu^{2/3}\frac{\max\{1,\|f_0\|_{H_{\bx,\te}^n}\}}{\nu^{ 2j/3+1/12}}\exp\lf\{C_0\nu^{2/3}t\rg\}+C_0\nu^{2/3}G.
\end{align*}
Solving the differential inequality yields that 
\begin{align}
\lf(\sum_{|i|=n-j}\|\pa_\bx^i\pa_\te^j f(t)\|_{L^2}^2\rg)^{1/2}\leq G(t)\leq C \frac{\max\{1,\|f_0\|_{H^n_{\bx,\te}}\}}{\nu^{2 j/3+1/12}}\exp\lf\{2C_0\nu^{2/3}t\rg\}.
\end{align}
This concludes the induction in $j$ and hence completes the induction in $n$.
\ifx  
\myb{Previous:
 Through $H^1$-energy estimate, we have that 
\begin{align}
\frac{1}{2}\frac{d}{dt}\|D_\bx f\|_{L^2}^2
= &-\nu\|\pa_\theta D_\bx f\|_{L^2}^2+\kappa\int\pa_\theta D_\bx f \cdot D_\bx (f L [f])\\
=& -\nu\|\pa_\theta D_\bx f\|_{L^2}^2+\kappa\int\pa_\theta \frac{|D_\bx f |^2}{2} L [f]+\kappa\int  f \pa_\theta D_\bx f \cdot D_\bx L [f]   \\
\leq &-\frac{\nu}{2}\|\pa_\theta D_\bx f\|_{L^2}^2+\frac{\kappa^2}{\nu}\|\Phi\pa_\theta\Psi\|_{L^2}^2\|f\|_{L^2}^2\|D_\bx f\|_{L^2}^2+\frac{\kappa^2}{\nu}\|D_\bx\Phi\Psi\|_{L^\infty}^2\|f\|_{L^2}^2.
\end{align}
Hence we have that 
\begin{align*}
\|D_\bx f(t)\|_{L^2}^2\leq C\|f_0\|_{H^1}^2\exp\lf\{C\frac{\kappa^2}{\nu}\lf(1+\frac{\kappa}{\nu}\rg) t\rg\}+C\int_{0}^t\exp\lf\{C\frac{\kappa^2}{\nu}\lf(1+\frac{\kappa}{\nu}\rg) (t-s)\rg\}\frac{\kappa^2}{\nu}\lf(1+\frac{\kappa}{\nu}\rg)ds.
\end{align*}
Since $\kappa\leq \nu^{5/6}\nu^{\zeta}$, $\zeta=\gamma-5/6>0$, we have that for $\nu$ small, 
\begin{align*}
\|D_\bx f(t)\|_{L^2}^2\leq& C\|f_0\|_{H^1}^2\exp\lf\{ \nu^{1/2+2\zeta} t\rg\}+\nu^{1/2+2\zeta}\int_{0}^t\exp\lf\{\nu^{1/2+2\zeta} (t-s)\rg\} ds\\
\leq&C(\|f_0\|_{H^1}^2+1)\exp\lf\{ \nu^{1/2+2\zeta} t\rg\}.
\end{align*}
Next we estimate the time derivative of $\|\pa_\te f\|_{L^2}^2$:
\begin{align}
\frac{1}{2}\frac{d}{dt}&\|\pa_\te f\|_{L^2}^2=-\nu\|\pa_\theta^2f\|_{L^2}^2+\kappa\int\pa_{\theta}^2  f \pa_\theta (f L [f])-\int \binom{-\sin\theta}{\cos\theta}\cdot D_\bx f \pa_\theta f\\
\leq &-\frac{\nu}{2}\|\pa_\theta D_\bx f\|_{L^2}^2+C\frac{\kappa^2}{\nu} \|\pa_\theta f\|_{L^2}^2+C\frac{\kappa^2}{\nu}\|f\|_{L^2}^2\|\pa_\theta f\|_{L^2}^2+\|D_\bx f\|_{L^2} \|\pa_\theta f\|_{L^2}\\
\leq&\frac{1}{4}\nu^{1/2+2\zeta} \|\pa_\theta f\|_{L^2}^2+\|D_\bx f\|_{L^2} \|\pa_\theta f\|_{L^2}.
\end{align}
Now we consider a function $y(t)=\sqrt{\|\pa_\theta f(t)\|_{L^2}^2+\ep}$, with $ \ep>0$ being a small and arbitrary number, (Following the argument in Alazard  \cite{Alazard21} page 101)
 and observe that the estimate above yields 
\begin{align}
\frac{d}{dt}y(t)^2\leq \frac{1}{2}\nu^{1/2+2\zeta} y(t)^2+2\|D_\bx f(t)\|_{L^2} y(t).
\end{align} Since $y(t)>0$, it is Lipschitz and we have that 
\begin{align}
\frac{d}{dt}y(t) \leq \frac{1}{4}\nu^{1/2+2\zeta} y(t) + \|D_\bx f(t)\|_{L^2} .
\end{align}Now we integrate in time and see that 
\begin{align*}
\|\pa_\te&  f(t)\|_{L^2}\leq y(t)\\
\leq & C(\|f_0\|_{H^1}+\sqrt\ep) \exp\lf\{\nu^{1/2+2\zeta} t\rg\}+C(\|f_0\|_{H^1}+1 )\int_{0}^t\exp\lf\{\frac{1}{4} \nu^{1/2+2\zeta}(t-s)+\frac{1}{2} \nu^{1/2+2\zeta}s\rg\} ds\\
\leq &C(\|f_0\|_{H^1}+1 )\frac{1}{\nu^{1/2+2\zeta}}\exp\lf\{\nu^{1/2+2\zeta} t\rg\}.
\end{align*}}
\fi
This concludes the proof.
\end{proof}

Next, we consider the estimate of the $f_\nq$ and $\lan f\ran$.
\begin{lem}Assume that the estimate \eqref{Hn_bnd} holds on the time interval $[0, \infty)$. Then, if $\nu$ is small enough, there exist constants $N_1=N_1(M,n)>0,\, N_2=N_2(M,n)>0,\, \delta_n>0$ such that  
\begin{align}\label{f_nq_Hn}
\|f_\nq\|_{H^n}\leq& C\nu^{-N_1}\|f_0\|_{H^M}\exp\{-\delta_n\nu^{1/2} t\}, \\ \label{fz_Hn}
\|\lan f\ran\|_{H^n}\leq&  C\nu^{-N_2}\|f_0\|_{H^M},\quad \forall n\leq M-1. 
\end{align}
\end{lem}
\begin{proof}The first estimate of $f_\nq$ is a natural consequence of the estimate \eqref{nl_ED} and the interpolation of $H^n$ functions
\begin{align*}
\|f_\nq\|_{\dot H^n}\leq C\|f_\nq\|_{L^2}^{\frac{M-n}{M}}\|f_\nq\|_{H^M}^{\frac{n}{M}}+\|f_\nq\|_{L^2}\leq C\nu^{-N_1}\|f_0\|_{H^M}\exp\lf\{C\nu^{2/3}t-\frac{M-n}{M}\delta\nu^{1/2} t\rg\}.
\end{align*}
Now we pick the $\nu $ small enough to derive \eqref{f_nq_Hn}. Given this bound, the derivation of \eqref{fz_Hn} is similar to \eqref{fz_est}. The main adjustment is to use the Gagliardo-Nirenberg inequality
\begin{align}
\|\lan f\ran\|_{\dot H^n}\leq C\|\lan f\ran\|_{L^2}^{\frac{1}{n+1}}\|\lan f\ran\|_{\dot H^{n+1}}^{\frac{n}{n+1}},\quad n\geq 1.
\end{align}
instead of the Nash inequality \eqref{Nash}. We omit further details for the sake of brevity. 
\end{proof}
\section{Energy Minimizers of the Spatially Homogeneous System}\label{sec:g_s} 
In this section, we present \cite{FrouvelleLiu12}'s idea in our setting to identify the critical points of the free energy functional $\eqref{Free_energy}_{\Psi (\cdot)=\sin(\cdot)}$.  We note that the stationary solution $g_s$ to the  equation \eqref{effctv_dym} satisfies the relation:
\begin{align}\label{D=0}
-\kappa \pa_\theta(g_s (\Psi_0*g_s))=\nu\pa_{\theta}^2g_s\quad\Rightarrow\quad -\kappa  g_s (\sin(\cdot)*g_s) =\nu\pa_{\theta}g_s+C.
\end{align} 
Now we introduce the quantity
\begin{align}
r:=\int_{-\pi}^\pi e^{-i\te}g_s(\te) d\te. 
\end{align}

We recall the definition of the Fourier transform and obtain that 
\begin{align*}
\lf(\sin*g_s\rg)(\te)=&\int_{-\pi}^\pi \frac{e^{i(\te-z)}-e^{-i(\te-z)}}{2i} g(z)dz=\frac{2\pi}{2i}(e^{i\theta}\wh g(1)-e^{-i\theta}\wh g(-1))\\
=&\frac{2\pi}{2i}(e^{i\theta}r-e^{-i\theta}\overline{r})=2\pi(\cos\theta \Im r+\sin\theta \Re r).
\end{align*} 
Now we plug this relation into \eqref{D=0}, to obtain the relation
\begin{align}
\pa_\theta  g_s
=-\frac{{2\pi}\kappa}{\nu} (\cos\theta \Im r+\sin\theta \Re{r})g_s+C.
\end{align} 
\myb{(Check?)} We can apply the Fourier transform and focus on the zero mode to see that $C=0.$  
Now we can use the integration factor to find the solution
\begin{align}\label{g_s}
g_s^{(r)}(\theta)
=\frac{1}{Z}\exp\lf\{\frac{ 2\pi\kappa}{\nu}(\cos\theta\Re r-\sin\theta\Im r)\rg\}=\frac{1}{Z}\exp\lf\{\frac{2\pi\kappa}{\nu}|r|\cos(\theta+\text{arg}(r))\rg\}.
\end{align}
Here $Z$ is the normalization factor to guarantee that $\|g_s\|_1=1$, i.e.,
\begin{align}
Z:=\int_{-\pi}^\pi \exp\lf\{\frac{2\pi\kappa}{\nu}|r|\cos(\theta)\rg \} d\theta=2\pi I_0\lf(\frac{2\pi\kappa}{\nu}|r|\rg).
\end{align}
Here $I_0$ is the modified Bessel function of the first kind.  
The constant state corresponds to $r=0.$  

To rigorously justify that \eqref{g_s} is indeed the solution to the stationary equation \eqref{D=0}, one needs to check that the resulting solution $g_s(\theta)$  indeed has Fourier coefficient $\wh g_s(1)=r.$ Now, we apply the Fourier transform
\begin{align}
\wh{g_s}(1)=&\frac{1}{2\pi I_0(2\pi \kappa| r|/\nu)} \int_{-\pi}^{\pi}\exp \lf\{\frac{2\pi\kappa |r|}{\nu}\cos(\theta+\arg r)\rg\}e^{-i(\theta+\arg r)} d\theta e^{i\arg r}\\
=&\frac{1}{2\pi I_0(2\pi \kappa| r|/\nu)} \int_{-\pi}^{\pi}\exp \lf\{\frac{2\pi\kappa |r|}{\nu}\cos(\theta)\rg\}\cos(\theta) d\theta e^{i\arg r}\\
=&\frac{I_1\lf(\frac{2\pi \kappa |r|}{\nu}\rg)}{2\pi I_0\lf(\frac{2\pi \kappa |r|}{\nu}\rg)}e^{i\arg r}=:F\lf(\frac{2\pi \kappa| r|}{\nu}\rg)e^{i\arg r}.
\end{align} 
Hence we have that the $g_s$ derived is indeed a solution if 
\begin{align}\label{comp_cond}
{2\pi}F\lf(\frac{2\pi \kappa |r|}{\nu}\rg)= {2\pi} {\frac{\nu}{2\pi \kappa  }\frac{2\pi \kappa }{\nu}}|r|.
\end{align}
Now we refer the readers to the paper \cite{FrouvelleLiu12} Proposition 3.3 (set $n=2$ there) to see that this compatibility condition has only trivial solution $|r|=0$ for $\kappa/\nu\leq 2$ and two distinct solutions ($|r_1|=0$, and $|r_2|=|r_2|(\kappa/\nu)$) for $\kappa/\nu>2$.
\myb{Details: Now we observe the following relation for the quotient $I_1(z)/I_0(z),\, z\geq 0.$ 
\begin{align*}
\frac{d}{dz}\frac{I_1(z)}{I_0(z)}=1-\frac{I_1(z)}{zI_0(z)}-\lf(\frac{I_1(z)}{I_0(z)}\rg)^2,\quad \frac{I_1(0)}{I_0(0)}=0.
\end{align*}
Hence we can see that the natural parameter range for the nontrivial stationary solution to exist is that $\frac{2\pi \kappa}{\nu}F'(0)>1$, which is $\kappa/\nu>2$. This is exactly the condition in the paper \cite{FrouvelleLiu12} and is exactly outside our linear stability range \eqref{lin_r}.}

\section{Derivation of the Kinetic Model}\label{App:Derivation}
In this section, we sketch the justification of the mean-field limit from \eqref{micro_eq} to the mesoscopic model \eqref{eq:bsc_1}. We use the main strategy in the paper \cite{BolleyCanizoCarrillo12}. First of all, we rewrite the system \eqref{micro_eq} in the form discussed in \cite{BolleyCanizoCarrillo12}. We recall that $\Psi(\theta)=\psi(\theta)\sin(\theta)$ for a smooth even function $\psi(\theta)$ on $[-\pi,\pi]=\mathbb{T}$ and define the velocity vectors $\bv^i=|\bv^i|(\cos(\theta^i),\sin(\theta^i))\in\rr^2,\, i\in\{1,2,\cdots, N\}$. Hence, we can define the function 
\begin{align}
\wt \psi(\bv^i,\bv^j)=\psi(\theta^i-\theta^j),\quad |\bv^i|,|\bv^j|\neq 0.
\end{align} 
Next, we recall the projection operator to the tangent space of $\mathbb{S}$:
\begin{align}
P(\bv)=I-\frac{\bv\otimes \bv}{|\bv|^2}.
\end{align}
 Then we can explicitly rewrite the equation \eqref{micro_eq} in terms of $(\bx^i, \bv^i)$: 
\begin{align}\label{micro_eq_form_2}
d\bx^i=&v(t)\bv^idt,\quad d\bv^i=\sqrt{2\nu} P(\bv^i)\circ dW^i-\kappa P(\bv^i)\lf(\frac{1}{N}\sum_{j=1}^N \Phi(\bx^i-\bx^j)(\bv^i-\bv^j)\wt\psi(\bv^i,\bv^j)\rg)dt,\\
&\bx^i(t=0)=\bx_0^i,\quad\bv^i(t=0)=\bv_0^i,\quad |\bv^i_0|=1,\quad i\in\{1,2,\cdots,N\}.
\end{align}
Here $\{dW^i\}_{i=1}^N$ are i.i.d. Brownian motions in $\rr^2$. Thanks to the discussion in \cite{BolleyCanizoCarrillo12}, the $|\bv^i|=1$ property is preserved overtime. To check the equivalence between the \eqref{micro_eq_form_2} and \eqref{micro_eq},  we recall from \cite{BolleyCanizoCarrillo12} that the Stratonovich noise $\sqrt{2\nu} P(\bv^i)\circ dW^i$ is equivalent to the diffusion process on $\mathbb{S}$, i.e., $\sqrt{2\nu}dB^i$ in \eqref{micro_eq}. Moreover, we observe that 
\begin{align*}
&-\kappa P(\bv^i)\lf(\frac{1}{N}\sum_{j=1}^N \Phi(\bx^i-\bx^j)(\bv^i-\bv^j)\wt\psi(\bv^i,\bv^j)\rg)\\
&=-\kappa\sum_{j=1}^N \Phi(\bx^i-\bx^j)\psi(\theta^i-\theta^j) \begin{bmatrix}
  \sin^2\theta^i & -\cos\theta^i\sin\theta^i \\
  -\cos\theta^i\sin\theta^i & \cos^2\theta^i
\end{bmatrix}\binom{\cos\theta^i-\cos\theta^j}{\sin\theta^i-\sin\theta^j}\\
&=\kappa\sum_{j=1}^N \Phi(\bx^i-\bx^j) \binom{-\sin\theta^i}{\cos\theta^i} \sin(\theta^j-\theta^i)\psi(\theta^j-\theta^i)\\
&=\kappa\sum_{j=1}^N \Phi(\bx^i-\bx^j) \Psi(\theta^j-\theta^i) \binom{-\sin\theta^i}{\cos\theta^i}. 
\end{align*}
Here we have used the fact that $\psi$ is even. Hence this term coincides with the corresponding alignment term in \eqref{micro_eq}. The above argument yields the equivalence between \eqref{micro_eq} and \eqref{micro_eq_form_2}. After developing the equivalence relation, one can follow the argument in \cite{BolleyCanizoCarrillo12} to take the mean-field limit ($N\rightarrow\infty$) in \eqref{micro_eq_form_2}. The resulting kinetic equation is equivalent to \eqref{eq:bsc_1}. We omit further details for the sake of brevity.

 \bibliographystyle{siam}
\bibliography{References_chronical_2022}

\end{document}